\documentclass[12pt]{amsart}
 \usepackage{amssymb}
 \usepackage{latexsym}
 \usepackage{graphicx}
 \usepackage{multicol}
 \usepackage{mathrsfs}
 \usepackage{pstricks, pst-node}
 \usepackage{young}
 \usepackage[vcentermath]{youngtab}
 \usepackage[all]{xy}
    \SelectTips{eu}{10 }     
    \everyxy={<2.5em,0em>:} 
 \usepackage{fancyhdr}      
 \newcounter{ctr}
 \theoremstyle{plain}
 \newtheorem{theorem}{Theorem}[section]
 
 \newtheorem*{lemma*}{Lemma}
 \newtheorem{lemma}[theorem]{Lemma}
 \newtheorem{corollary}[theorem]{Corollary}
 \newtheorem{proposition}[theorem]{Proposition}
 \newtheorem{conjecture}[theorem]{Conjecture}

 \theoremstyle{definition}

 \newtheorem{remark}[theorem]{Remark}
 \newtheorem{example}[theorem]{Example}

 \newcommand{\CC}{\ensuremath{\mathbb{C}}}

 \newcommand{\G}{\ensuremath{\mathscr{G}}}
 \renewcommand{\H}{\ensuremath{\mathscr{H}}}
 
 \renewcommand{\hom}{\text{\rm Hom}}

 \newcommand{\id}{\text{\rm Id}}
 
 \newcommand{\im}{\text{\rm im\,}}
 
 \newcommand{\ind}{\text{\rm Ind}}

 \renewcommand{\L}{\ensuremath{\mathscr{L}}}

 \newcommand{\Mod}{\ensuremath{\mathbf{Mod}}}

 \newcommand{\ZZ}{\ensuremath{\mathbb{Z}}}
\newcommand{\leftexp}[2]{{\vphantom{#2}}^{#1}{#2}}
\newcommand{\leftsub}[2]{{\vphantom{#2}}_{#1}{#2}}

\newcommand{\lj}[2]{{#1}_{#2}}
\newcommand{\lJ}[2]{{#1}^{#2}}
\newcommand{\rj}[2]{\leftsub{#2}{#1}}
\newcommand{\rJ}[2]{\leftexp{#2}{#1}}
\newcommand{\Res}{\text{\rm Res}}
\renewcommand{\S}{\ensuremath{\mathcal{S}}}
\newcommand{\tsr}{\ensuremath{\otimes}}
\newcommand{\C}{\ensuremath{C'}} 
\newcommand{\tB}{\ensuremath{{\tilde{C'}}\negmedspace}}
\newcommand{\tP}{\ensuremath{\tilde{P}}}
\newcommand{\tE}{\ensuremath{\widetilde{E}}}
\newcommand{\tF}{\ensuremath{\widetilde{F}}}
\newcommand{\tT}{\ensuremath{\widetilde{T}}}
\newcommand{\hE}{\ensuremath{\widehat{E}}}
\newcommand{\hF}{\ensuremath{\widehat{F}}}
\newcommand{\bt}{\small \ensuremath{\boxtimes}}
\newcommand{\br}[1]{\ensuremath{\overline{#1}}}
\renewcommand{\u}{\ensuremath{u}}  
\newcommand{\ui}{\ensuremath{u^{-1}}} 
\newcommand{\T}{\ensuremath{\tilde{T}}}

\newcommand{\Cell}{\ensuremath{\mathfrak{C}}}
\newcommand{\Tab}{\ensuremath{\mathcal{T}}}
\newcommand{\ic}{\ensuremath{\mathcal{IC}}}

\newcommand{\eW}{\ensuremath{{W_e}}}
\newcommand{\pW}{\ensuremath{{W_e^+}}}
\newcommand{\eH}{\ensuremath{\widehat{\H}}} 
\newcommand{\pH}{\ensuremath{\widehat{\H}^{+}}} 
\newcommand{\pY}{\ensuremath{Y_{\geq 0}}}

\newcommand{\eS}{\widehat{\S}}

\newcommand{\aW}{\ensuremath{{W_a}}}

\newcommand{\jdt}{\text{\rm jdt}}

\newcommand{\sh}{\text{\rm sh}}
\newcommand{\evac}{\text{\rm evac}}

\newcommand{\tsym}{\ensuremath{tS^{2}V}}
\newcommand{\tsymred}{\ensuremath{Z^2}}
\newcommand{\tsymredaff}{\ensuremath{\widehat{tS^{2}}_{\text{red}}V}}
\newcommand{\symred}{\ensuremath{S^2_{\text{red}}V}}

\newcommand{\be}{\begin{equation}}
\newcommand{\ee}{\end{equation}}

\renewcommand{\ng}{\text{-}}
\newcommand{\mone}{\ensuremath{\text{\ng 1}}}
\newcommand{\mfour}{\ensuremath{\text{\ng 4}}}

\newcommand{\unit}{\alpha}
\newcommand{\counit}{\beta}

\newcommand{\JW}{\ensuremath{\leftexp{J}{W}_{2}}}
\newcommand{\WJW}{\ensuremath{W_{1}\stackrel{J}{\times} {W}_{2}}}
\newcommand{\HJH}{\ensuremath{\H_1\tsr_J \H_{2}}}
\newcommand{\WJJW}{\ensuremath{W_{1}\stackrel{J_1}{\times} \ldots\stackrel{J_{d-1}}{\times}W_{d}}}
\newcommand{\HJJH}{\ensuremath{\H_{1}\tsr_{J_1}\ldots\tsr_{J_{d-1}}\H_{d}}}

\begin{document}
\author{Jonah Blasiak}
\title{$W$-graph versions of tensoring with the $\S_n$ defining representation}

\begin{abstract}
We further develop the theory of inducing $W$-graphs worked out by Howlett and Yin in \cite{HY1}, \cite{HY2}, focusing on the case $W = \S_n$.  Our main application is to give two $W$-graph versions of tensoring with the $\S_n$ defining representation $V$, one being $\H \tsr_{\H_J} -$ for $\H, \H_J$ the Hecke algebras of $\S_n, \S_{n-1}$ and the other $(\pH \tsr_{\H} -)_1$, where $\pH$ is a subalgebra of the extended affine Hecke algebra and the subscript signifies taking the degree 1 part.  We look at the corresponding $W$-graph versions of the projection $V \tsr V \tsr - \to S^2 V \tsr -$. This does not send canonical basis elements to canonical basis elements, but we show that it approximates doing so as the Hecke algebra parameter $\u \to 0$.  We make this approximation combinatorially explicit by determining it on cells. Also of interest is a combinatorial conjecture stating the restriction of $\H$ to $\H_J$ is ``weakly multiplicity-free'' for $|J| = n-1$, and a partial determination of the map $\H \tsr_{\H_J} \H \xrightarrow{\counit} \H$ on canonical basis elements, where $\counit$ is the counit of adjunction.
\end{abstract}

\maketitle

\section{Introduction}
The polynomial ring $R := \CC[x_1,\ldots,x_n]$ is well understood as a $\CC\S_n$-module, but how this $\CC\S_n$-module structure is compatible with the structure of $R$ as a module over itself is not.
This work came about from an attempt to construct a combinatorial model for $R$ as a $\CC\S_n$-submodule that takes into account multiplication by the $x_i$.  The hope is that such a model would lead to a better understanding of the Garsia-Procesi modules, particularly, the combinatorics of cyclage and catabolism.  We also might hope to find modules corresponding to the $k$-atoms of Lascoux, Lapointe, and Morse and uncover combinatorics that governs them.

Such a model might look something like this: decompose the tensor algebra $TV$ into canonically chosen
irreducible $\CC \S_n$-submodules,  where $V$ is the  degree 1 part of $R$.
 Define a poset in which an irreducible $E'$ is less than an irreducible $E$ if $E' \subseteq V\tsr E$.  Somehow project this picture onto a canonical decomposition of $R$ into $\CC\S_n$-irreducibles.  Lower order ideals of the projected poset would correspond to $\CC\S_n$-modules that are also $R$-modules.  Edges would be controlled by a local rule saying that a path of length two $(E,E'), (E',E'')$ must satisfy $E'' \subseteq S^2V \tsr E$.

The main results of this paper are a first step towards this approach; further work will appear in \cite{B}.  To obtain a nice decomposition of $TV$ and $R$ into irreducibles, we replace $\CC \S_n$ with the Hecke algebra $\H$ of $W = \S_n$ and apply the theory of canonical bases. The functor $V \tsr -$ is replaced by $\H \tsr_{\H_J} -$, $J = \{ s_2, \ldots, s_{n-1} \} \subseteq S$, $S$ the simple reflections of $W$. We are naturally led to a construction that takes an $\H_J$-module $E$ coming from a $W_J$-graph and produces a $W$-graph structure on $\H \tsr_{\H_J} E$. This construction of inducing $W$-graphs, found independently by the author, is due to Howlett and Yin \cite{HY1}. We spend a good deal of this paper (\textsection\ref{s Hecke algebra} -- \textsection\ref{s tableau combinatorics}) developing this theory, proving some basic results of interest for their own sake as well as for this application.

Once this groundwork is laid, we can form a $W$-graph version of $TV \tsr E$, $TV$ being the tensor algebra of $V$, for any $\H$-module $E$ coming from a $W$-graph. We can then try to project this onto a $W$-graph version of $SV \tsr E = R \tsr E$.  This is even interesting for $T^2 V$ and $S^2 V$ and is what we focus on in this paper. Define $T^2_\text{red} V := \ZZ \{ x_i \tsr x_j : i \neq j\}$ and $S^2_\text{red} V := \ZZ \{ x_i \tsr x_j + x_j \tsr x_i: i \neq j\}$. We show in Proposition \ref{p red notred} that our $W$-graph version of $T^2 V \tsr E$ has a cellular decomposition into $\tF^2 := \H \tsr_{J \backslash s_2} E$ and $\H \tsr_J E$, which at $\u=1$ become $T^2_\text{red} V \tsr E$ and $V \tsr E$.  There is a canonical map (\ref{e reduced v tsr v to reduced s^2})
$$ \tF^2 \xrightarrow{\tilde{\counit}} \H \tsr_{S \backslash s_2} E,$$
specializing at $\u =1$ to the projection $T^2_\text{red} V \tsr E \to S^2_\text{red} V \tsr E$.   The map $\tilde{\counit}$ does not send canonical basis elements to canonical basis elements, but it approximates doing so as the Hecke algebra parameter $\u \to 0$ (Corollary \ref{c approx at u=1}).  This partitions the canonical basis of $\tF^2$ into two parts--the approximate kernel, which we refer to as combinatorial wedge, and the  approximate inverse image of the canonical basis of $\H \tsr_{S \backslash s_2} E$, which we refer to as combinatorial reduced sym.  Theorem \ref{t combinatorial sym} determines this partition in terms of cells.

We also consider a $W$-graph version of tensoring with $V$ coming from the extended affine Hecke algebra. This mostly parallels the version just described, but there are some interesting differences.  Most notably, the combinatorics of this $W$-graph version of the inclusion $T^2_\text{red} V \tsr E \to V \tsr V \tsr E$ is transpose to that of the other; compare Theorems \ref{t combinatorial sym} and \ref{t combinatorial sym affine}.

This paper is organized mainly in order of decreasing generality. We begin in \textsection\ref{s Hecke algebra} by introducing the Hecke algebra, $W$-graphs, and the inducing $W$-graph construction. We reformulate some of this theory using the formalism of IC bases as presented in \cite{Du}. This has the advantage of avoiding explicit calculations involving Kazhdan-Lusztig polynomials, or rather, hides these calculations in the citations of \cite{HY1}, \cite{HY2}. This allows us to focus more on cells and cellular subquotients. In \textsection\ref{s H1H2} we specialize to
the case where $W$-graphs come from iterated induction from the regular representation. In this case we prove that all left cells are isomorphic to those occurring in the regular representation of $W$ (Theorem \ref{t HJHcells}). Next, in \textsection\ref{s tableau combinatorics}, we review the combinatorics of cells in the case $W = \S_n$ . As was first observed in \cite{R}, there is a beautiful connection between the Littlewood-Richardson rule and the cells
of an induced module $\H \tsr_{\H_J} E$ (Proposition \ref{p cells of induce}).  The combinatorics of the cells
of the restriction $\Res_{\H_J} \H$ is less familiar; see Conjecture \ref{c weak mult free}. Sections \ref{s computation c_w} and \ref{s canonical maps} give a nice result about how canonical basis elements behave under the projection $\H \tsr_J \H \to \H$.  The remaining sections \ref{s tensor V}, \ref{s decomposing VVE}, and \ref{s combinatorial approximation S2} contain our main results just discussed.



\section{IC bases and inducing $W$-graphs}

\label{s Hecke algebra}
\subsection{}
We will use the following notational conventions in this paper.  If $A$ is a ring and $S$  is a set, then $A S$ is a free $A$-module with basis $S$ (possibly endowed with some additional structure, depending on context).  Elements of induced modules $\H \tsr_{\H_J} E$ will be denoted $h \bt e$ to distinguish them from elements of a tensor product over $\ZZ$, $F \tsr_\ZZ E$, whose elements will be denoted $f \tsr e$. The symbol $[n]$ is used for the set $\{1,\ldots,n\}$ and also for the $\u$-integer (defined below), but there should be no confusion between the two.

\subsection{}

\label{ss coxeter group}
Let $W$ be a Coxeter group and $S$ its set of simple reflections. The \emph{length} $\ell(w)$ of $w \in W$ is the minimal $l$ such that $w=s_1\ldots s_l$ for some $s_i\in S$. If $\ell(uv)=\ell(u)+\ell(v)$, then $uv = u\cdot v$ is a \emph{reduced factorization}. The notation $L(w) = \{s\in S : sw < w\}, R(w) = \{s\in S : ws < w\}$ will be used for the left and right descent sets of $w$.

For any $J\subseteq S$, the \emph{parabolic subgroup} $W_J$ is the subgroup of $W$ generated by $J$. Each left (resp. right) coset $wW_J$ (resp. $W_Jw$) contains an unique element of minimal length called a minimal coset representative. The set of all such elements is denoted $W^J$ (resp. $\leftexp{J}W$). For any $w \in W$, define $\lJ{w}{J}$, $\rj{w}{J}$ by
\be w=\lJ{w}{J} \cdot \rj{w}{J},\ \lJ{w}{J} \in W^J,\ \rj{w}{J} \in W_J.\ee
Similarly, define $\lj{w}{J}$, $\rJ{w}{J}$ by
\be w= \lj{w}{J} \cdot \rJ{w}{J},\ \lj{w}{J} \in W_J,\ \rJ{w}{J} \in \leftexp{J}W.\ee

\subsection{}
Let $A = \ZZ[\u,\ui]$ be the ring of Laurent polynomials in the indeterminate $\u$, $A^{-}$ (resp. $A^{+}$) be the subring $\ZZ[\ui]$ (resp. $\ZZ[\u]$), and $\br{\cdot}:A\to A$ be the involution given by $\br{u} = \ui$. The \emph{Hecke algebra} $\H$ of $W$ is the free $A$-module with basis $\{T_w :\ w\in W\}$ and relations generated by
\be \label{e Hecke algebra def} \begin{array}{ll}T_uT_v = T_{uv} & \text{if } uv = u\cdot v\ \text{is a reduced factorization},\\
(T_s - \u)(T_s + \ui) = 0 & \text{if } s\in S.\end{array}\ee

For each $J\subseteq S$, $\H_J$ denotes the subalgebra of $\H$ with $A$-basis $\{T_w:\ w\in W_J\}$, which is also the Hecke algebra of $W_J$.

The involution, $\br{\cdot}$, of $\H$ is the additive map from $\H$ to itself extending the involution $\br{\cdot}$ on $A$ and satisfying $\br{T_w} = T_{w^{-1}}^{-1}$. Observe that $\br{T_{s}} = T_s^{-1} = T_s + \ui - u$ for $s \in S$. Some simple $\br{\cdot}$-invariant elements of $\H$ are $\C_\text{id} := T_\text{id}$ and $\C_s := T_s + \ui = T_s^{-1} + u$, $s\in S$. The $\br{\cdot}$-invariant $\u$-integers are $[k] := \frac{\u^k - \u^{-k}}{\u - \ui} \in A$.

\subsection{}
\label{ss IC basis}
\newcommand{\ord}{\ensuremath{\prec}}
Before introducing $W$-graphs and the Kazhdan-Lusztig basis, we will discuss a slightly more general setup for canonical bases.
The presentation here follows Du \cite{Du}. This formalism originated in \cite{KL} and was further developed by Lusztig and Kashiwara (see the references in \cite{Du}).

Given any $A$-module $E$ (no Hecke algebra involved), we can try to construct a \emph{canonical basis} or \emph{IC basis} from a \emph{standard basis} and involution $\br{\cdot}:E\to E$.  Let  $\{t_i:\ i\in I\}$ be an $A$-basis of $E$ (the standard basis) for some index set $I$ and assume the involution $\br{\cdot}$ \emph{intertwines} the involution $\br{\cdot}$ on $A$: $\br{at}=\br{a}\br{t}$ for any $a\in A$, $t\in E$. Define the lattice $\L$ to be $A^{-}\{t_i:\ i\in I\}$.  If there exists a unique $\br{\cdot}$-invariant basis $\left\{c_i : i\in I\right\}$ of the free $A^{-}$-module $\L$ such that $c_i \equiv t_i \mod \ui\L$, then $\left\{c_i : i\in I\right\}$ is an IC basis of $E$, denoted
\be \ic_E(\{t_i: i \in I\},\br{\cdot}). \ee
\begin{theorem}[Du \cite{Du}]\label{t IC basis}
With the notation above, if $(I, \ord)$ is a poset such that for all $j \in I$, $\{i \in I:i \ord j\}$ is finite and $\br{t_j} \equiv t_j \mod A \{t_i: i\ord j\}$, then the IC basis $\ic_E(\{t_i: i \in I\},\br{\cdot})$ exists.
\end{theorem}
In the remainder of this paper, $\br{\cdot}$ will be clear from context so will be omitted from the $\ic()$ notation.
An observation that will be used in \textsection \ref{s cells} and \textsection \ref{s H1H2} is that this construction behaves well with taking lower order ideals.
\begin{proposition}\label{p ic lower ideal}
  With the notation of Theorem \ref{t IC basis}, if $I'$ is a lower order ideal of $I$ and $E' := A\{t_i:\ i\in I'\}$ , then $$  \ic_{E'}(\{t_i:\ i\in I'\}) = \{c_i: i \in I'\} \subseteq \ic_E(\{t_i:\ i\in I\} )$$
\end{proposition}
\begin{proof}
The poset $I'$ and the involution $\br{\cdot}$ restricted to $E'$ satisfy the necessary hypotheses so that Theorem \ref{t IC basis} applies.  Label the resulting IC basis by $d_i$, $i\in I'$ and put $\L' = A^{-}\{t_i:\ i\in I'\}$.  Then $d_i \equiv t_i \mod \ui\L'$ for $i\in I'$ certainly implies $d_i \equiv t_i \mod \ui\L$.  Uniqueness of the IC basis then implies $d_i = c_i \ (i \in I')$.
\end{proof}

We now come to the main construction studied in this paper. Let $E$ be an $\H_J$-module with an involution $\br{\cdot} : E\to E$ intertwining $\br{\cdot}$ on $\H_J$ ($\br{h e} = \br{h} \br{e}$ for all $h \in \H_J$ and $e \in E$).  Suppose $\Gamma$ is a $\br{\cdot}$-invariant $A$-basis of $E$.  Put $\tE = \H \tsr_{\H_J} E$.  We will apply Theorem \ref{t IC basis} to $\tE$ with standard basis $\tT := \{ \T_\mathbf{z} : \mathbf{z} \in W^J \times \Gamma\}$, where $\T_{w,\gamma} := T_w \bt \gamma$.  The lattice $\L$ is then $A^{-}\tT$.  Define the involution $\br{\cdot}$ on $\tE$ from the involutions on $E$ and $\H$:
\be \br{h\bt e} = \br{h}\bt \br{e}, \text{ for every $h\in\H, e\in E$}. \ee
It is easy to check (and is done in \cite{HY1}) that the definition of $\br{\cdot} : \tE \to\tE$ is sound, that it's an involution and intertwines the involution $\br{\cdot}$ on $\H$.

Let $\ord$ be the partial order on $W^J\times \Gamma$ generated by the rule: $(w',\gamma')\ord(w,\gamma)$ if $\T_{w',\gamma'} $ appears with non-zero coefficient in $(\br{T_w}-T_w) \bt \gamma$ expanded in the basis $\tT$.  Since $\br{T_w} - T_w$ is an $A$-linear combination of $T_x$ for $x < w$, it is easy to see that $\br{\T_{w,\gamma}} - \T_{w,\gamma}$ ($w\in W^J, \ \gamma \in \Gamma$) is an $A$-linear combination of $\{ \T_{x , \delta} : x<w, \delta \in \Gamma \}$, so the definition of $\ord$ is sound.   To see that $D_{w,\gamma} := \{(w',\gamma'):(w',\gamma')\preceq(w,\gamma)\}$ is finite, induct on $\ell(w)$.  The set $D_{w,\gamma}$ is the union of $\{ (w, \gamma) \}$ and $D_{w',\gamma'}$ over those $(w',\gamma')$ such that $\T_{w',\gamma'}$ appears with non-zero coefficient in $(\br{T_w}-T_w) \bt \gamma$, each of which is finite by induction.

Thus Theorem \ref{t IC basis} applies and we obtain a canonical basis $\Lambda = \ic_{\tE}(\tT) = \{\tB_{w,\gamma} : w \in W^J, \gamma \in \Gamma\}$ of $\tE$.  This is one way of proving the following theorem that is Theorem 5.1 in \cite{HY1} (there they use the basis $C_{w,\gamma}$ that is $\equiv\T_{w,\gamma}\mod \u A^+\tT)$.

\begin{theorem}[Howlett, Yin \cite{HY1}]
\label{t HY canbas exists}
There exists a unique $\br{\cdot}$-invariant basis $\Lambda=\{\tB_{w,\gamma} : w \in W^J, \gamma \in \Gamma\}$ of $\tE$ such that $\tB_{w,\gamma} \equiv \T_{w,\gamma}\mod \ui \L$.\end{theorem}

Applied to $J=\emptyset$ and $\Gamma$ the free $A$-module of rank one, this yields the usual Kazhdan-Lusztig basis $\Gamma_W:=\{\C_w:w\in W\}$ of $\H$.

\subsection{}
In \cite{KL}, Kazhdan and Lusztig introduce $W$-graphs as a combinatorial structure for describing an $\H$-module with a special basis. A $W$-graph consists of a vertex set $\Gamma$, an edge weight $\mu(\delta,\gamma)\in \ZZ$ for each ordered pair $(\delta,\gamma)\in\Gamma\times\Gamma$, and a descent set $L(\gamma) \subseteq S$ for each $\gamma\in\Gamma$. These are subject to the condition that $A\Gamma$ has a left $\H$-module structure given by
\begin{equation}\label{Wgrapheq}\C_{s}\gamma = \left\{\begin{array}{ll} [2] \gamma & \text{if}\ s \in L(\gamma),\\
\sum_{\substack{\{\delta \in \Gamma : s \in L(\delta)\}}} \mu(\delta,\gamma)\delta & \text{if}\ s \notin L(\gamma). \end{array}\right.\end{equation}

We will use the same name for a $W$-graph and its vertex set.  If an $\H$-module $E$ has an $A$-basis $\Gamma$ that satisfies (\ref{Wgrapheq}) for some choice of descent sets, then we say that $\Gamma$ gives $E$ a \emph{$W$-graph structure}, or $\Gamma$ is a $W$-graph on $E$.

It is convenient to define two $W$-graphs $\Gamma,\Gamma'$ to be isomorphic if they give rise to isomorphic $\H$-modules with basis. That is, $\Gamma\cong\Gamma'$ if there is a bijection $\alpha:\Gamma\to\Gamma'$ of vertex sets such that $L(\alpha(\gamma)) = L(\gamma)$ and $\mu(\alpha(\delta),\alpha(\gamma)) = \mu(\delta,\gamma)$ whenever $L(\delta)\not\subseteq L(\gamma)$.

Given a $W$-graph $\Gamma$, we always have an involution
\be \label{e W-graph inv}
\br{\cdot}: A \Gamma \to A \Gamma, \text{ with } \br{\gamma} = \gamma \text{ for every } \gamma \in \Gamma,\ee
and extended $A$-semilinearly using the involution on $A$.  It is quite clear from (\ref{Wgrapheq}) (and checked in \cite{HY1}) that this involution intertwines $\br{\cdot}$ on $\H$.

\subsection{}\label{ss induced W graph}
Now let $\Gamma$ be a $W_J$-graph, $E = A \Gamma$, and $\br{\cdot} : E \to E$ be as just mentioned in (\ref{e W-graph inv}).  Then we are in the setup of \textsection \ref{ss IC basis} except $\Gamma$ is a $W_J$-graph instead of any $\br{\cdot}$-invariant basis of $E$.  Maintaining the notation of \textsection \ref{ss IC basis}, let $\Lambda = \ic_{\tE}(\tT) = \left\{\tB_{w,\gamma} :\ w\in W^J, \gamma\in\Gamma\right\}$.  As would be hoped, $\Lambda$ gives $\tE$ a $W$-graph structure.



Define $\tP_{x,\delta,w,\gamma}$ by the formula
\be\tB_{w,\gamma} = \sum\limits_{(x,\delta)\in W^J\times\Gamma} \tP_{x,\delta,w,\gamma}\ \T_{x,\delta}.\ee

For every $(x,\delta),(w,\gamma)\in W^J \times \Gamma$ define
\be\label{e mudef}
\mu(x,\delta,w,\gamma) = \left\{\begin{array}{ll}
\text{coefficient of}\ \ui \ \text{in}\  \tP_{x,\delta,w,\gamma} & \text{if } x < w,\\
\mu(\delta,\gamma) & \text{if } x = w, \\
1 & \text{if } x = sw,\ x > w,\ s\in S,\ \delta = \gamma,\\
0 & \text{otherwise}. \end{array}\right.\ee%
Also define $L(w,\gamma)=L(w)\cup \{s\in S : sw = wt,\ t\in L(\gamma)\}$.
Now we can state the main result of Howlett and Yin.
\begin{theorem}[{\cite[Theorem 5.3]{HY1}}] \label{t HY} With $\mu$ and $L$ as defined above, $\Lambda$ gives $\tE = \H\tsr_{\H_J} A\Gamma$ a $W$-graph structure.
\end{theorem}

We will often abuse notation and refer to a module when we really mean the $W$-graph on that module, but there should be no confusion as there will never be more than one $W$-graph structure on a given module.  We will use the notation $\H \tsr_{\H_J} \Gamma$ to mean the $\Lambda$ in this theorem, in case we want refer to its vertex set or to emphasize the $W$-graph rather than the module.

\begin{remark}A $W$-graph is \emph{symmetric} if it is isomorphic to a $W$-graph with $\mu(x,w) = \mu(w,x)$ for all vertices $x,w$. The $W$-graph $\Gamma_W$ on the regular representation of $\H$ is symmetric. The $W$-graph $\Lambda$ defined above is symmetric if and only if $\Gamma$ is symmetric, although this is not obvious from the definition of $\mu$ (\ref{e mudef}). In \cite{HY1}, the $W$-graph for $\Lambda$ is defined so that it is clearly symmetric, and then it is proved later that it is isomorphic to the $W$-graph $\Lambda$ defined here.
\end{remark}

\subsection{}
\label{s cells}
Let $\Gamma$ be a $W$-graph and put $E=A\Gamma$. The preorder $\leq_{\Gamma}$ on the vertex set $\Gamma$ is generated by
\be \delta\leq_{\Gamma}\gamma \begin{array}{c}\text{if there is an $h\in\H$ such that $\delta$ appears with non-zero}\\ \text{coefficient in the expansion of $h\gamma$
in the basis $\Gamma$}. \end{array}
\ee
Equivalence classes of $\leq_{\Gamma}$ are the \emph{left cells} of $\Gamma$, or just \emph{cells} since we will almost exclusively work with left cells. Sometimes we will speak of the cells of $E$ or the preorder on $E$ to mean that of $\Gamma$, when the $W$-graph $\Gamma$ is clear from context.
 A \emph{cellular submodule, quotient, or subquotient} of $E$ is a submodule, quotient, or subquotient of $E$ that is spanned by a subset of $\Gamma$ (and is necessarily a union of cells). We will abuse notation and sometimes refer to a cellular subquotient by its corresponding union of cells.

We will give one result about cells in the full generality of \textsection\ref{ss induced W graph} before specializing $W$ and the $W_J$-graph $\Gamma$. Let $D$ be a cellular submodule of $E$ spanned by a subset $\Gamma_D$ of $\Gamma$ and $p : E\to E/D$ the projection. Put $\Gamma_{E/D} = p(\Gamma\backslash \Gamma_D)$. The $W_J$-graph $\Gamma$ yields $W_J$-graphs $\Gamma_D$ on $D$ and $\Gamma_{E/D}$ on $E/D$.  The involution $\br{\cdot}$ on $E$ restricts to one on $D$ and projects to one on $E/D$; elements of $\Gamma_D$ (resp. $\Gamma_{E/D}$) are fixed by the involution $\br{\cdot}$ on $D$ (resp. $E/D$). Since $\H$ is a free right $\H_J$-module, we have the exact sequence
\be
\xymatrix{0 \ar[r] & \H\tsr_JD \ar[r] & \H\tsr_JE \ar[r]_<<<<{\tilde{p}}& \H\tsr_JE / D \ar[r]& 0,}
\ee
where the shorthand $\H\tsr_{J}E := \H\tsr_{\H_J}E$ will be used here and from now on. In other words, inducing commutes with taking subquotients. It is also true that inducing and taking canonical bases commutes with taking cellular subquotients:

\begin{proposition}\label{CBSubquotientprop} With the notation above and that of \textsection\ref{ss induced W graph}, let $\tT_D = \left\{\T_{w,\gamma} : w\in W^J, \gamma\in \Gamma_D\right\}$ and $\tT_{E/D} = \left\{T_w\bt \gamma: w\in W^J, \gamma\in \Gamma_{E/D}\right\}$.  Then
 \begin{list}{(\roman{ctr})}{\usecounter{ctr} \setlength{\itemsep}{1pt} \setlength{\topsep}{2pt}}
\item $\ic_{\H\tsr_J D}\left(\tT_D \right) = \left\{\tB_{w,\gamma} : w\in W^J, \gamma\in \Gamma_D \right\}\subseteq \ic_{\tE}(\tT)$,

\item $\ic_{\H\tsr_J E/D}\left(\tT_{E/D}\right) = \left\{\tilde{p}(\tB_{w,\gamma}) : w\in W^J, \gamma\in \Gamma\backslash \Gamma_D \right\}\subseteq \tilde{p}\left(\ic_{\tE}(\tT)\right)$.
\end{list}
\end{proposition}

\begin{proof}
Statement (i) is actually a special case of Proposition \ref{p ic lower ideal}.  From the definition of $\ord$ in \textsection \ref{ss IC basis} we can see that $W^J \times \Gamma_D$ is a lower order ideal of $W^J \times \Gamma$.

We prove (ii) directly.  The lattice $\L_{E/D} := A^{-}\tT_{E/D}$ is the quotient $\L / \L_D = \tilde{p}(\L)$. Therefore, given $w\in W^J$ and $\gamma\in \Gamma\backslash \Gamma_D$, we have
\be\tilde{p}(\tB_{w,\gamma}) = \tilde{p}(T_w \bt \gamma + \ui x) \equiv \tilde{p}(T_w \bt \gamma) = T_w \bt p(\gamma),\ee
where $x$ is some element of $\L$ and the congruence is $\mod \ui \L_{E/D}$.  By definition, $p(\gamma) \in \Gamma_{E/D}$ so $\tilde{p}(\tB_{w,\gamma})$ is the element of $\ic_{\H\tsr_J E/D}\left(\tT_{E/D}\right)$ congruent to $T_w \bt p(\gamma)\mod\ui\L_{E/D}$.
\end{proof}

This proposition is essentially \cite[Theorem 4.3]{HY2}, though the proof here is different. It also appears in \cite[Theorem 1]{G1} in the case that $\Gamma = \Gamma_{W_J}$ (the usual $W_J$-graph on $\H_J$) but in the generality of unequal parameters.
\section{Iterated induction from the regular representation}
\label{s H1H2}
In this paper we will primarily be interested in the case where $E$ is obtained by some sequence of inductions and restrictions of the regular representation of a Hecke algebra, or subquotients of such modules. In this section, let $\tE$ denote $\H_1 \tsr_J E$, where $E = A \Gamma, \Gamma = \Gamma_{W_2}$ unless specified otherwise.

\subsection{}

Suppose we are given Coxeter groups $W_1$, $W_2$ with simple reflections $S_1,S_2$ and a set $J$ with inclusions $i_k: J\to S_k, k=1,2$ such that ${(W_1)}_{i_1(J)} \cong {(W_2)}_{i_2(J)}$ as Coxeter groups. Define the set
\be\WJW:=
\left\{(w_1,w_2) : w_1\in W_1, w_2\in W_2\right\} \big/ \langle (w_1w,w_2)\sim(w_1,ww_2) : w\in W_J\rangle,\ee
where $W_J := {W_1}_J \cong {W_2}_J$. The set $\WJW$ can also be identified with any of $W_1\times \JW$, ${W_1}^J\times W_2$, or  $W_1^{J}\times W_J\times \JW$.  These sets label canonical basis elements of Hecke algebra modules obtained by inducing from the regular representation just as a Coxeter group labels the canonical basis elements of its regular representation.

The material that follows in this subsection is somewhat tangent from our main theme, but we include it for completeness. We omit the details of proofs.

The set $\WJW$ comes with a left action by $W_1$, a length function, and a partial order generalizing the Bruhat order, as described in the following proposition.

\begin{proposition}\label{W-set prop}Let $(w_1,w_2)\in \WJW$. The set $\WJW$ comes equipped with \begin{list}{(\roman{ctr})} {\usecounter{ctr} \setlength{\itemsep}{1pt} \setlength{\topsep}{2pt}}
\item A left action by $W_1 :\  x\cdot (w_1, w_2) = (xw_1, w_2)$,
\item a length function: $\ell(w_1,w_2) = \ell(w_1)+\ell(w_2)$ whenever $w_1\in {W_1}^J$,
\item a partial order: $(w_1',w_2')\leq (w_1,w_2)$, whenever there exists $(w_1'',w_2'')\sim (w_1',w_2')$ such that $w_i''\leq w_i,\ w_i',w_i'' \in {W}_i, i = 1,2,$ and $w_1\in {W_1}^J$.
\end{list}
\end{proposition}

\begin{proposition}
The $W_1$-graph $\tE$ is bipartite in the sense of \cite[Definition 3.1]{HY2}. Moreover, if $\mathbf{z}, \mathbf{z'} \in \WJW$, and $\ell(\mathbf{z})-\ell(\mathbf{z'})$ is even (resp. odd), then $\tP_{\mathbf{z'},\mathbf{z}}$ involves only even (resp. odd) powers of $\u$.
\end{proposition}
\begin{proof}
This follows from \cite[Proposition 3.2]{HY2}.
\end{proof}

\begin{proposition}
The $W_1$-graph $\tE$ is ordered in the sense of \cite[Definition 1.1]{HY2}. Stronger, $\WJW$ has a partial order from Proposition 2.2 of \cite{HY2} using the Bruhat order on $W_2$, and this agrees with $\leq$ of Proposition \ref{W-set prop}. Therefore if $\mathbf{z, z'} \in \WJW$ and $\tP_{\mathbf{z', z}} \neq 0$, then $\mathbf{z'} \leq \mathbf{z}$.
\end{proposition}
\begin{proof}
Showing the partial orders from \cite{HY2} and Proposition \ref{W-set prop} are equal takes some work. The rest is a citation of results in \cite{HY2}.
\end{proof}

\subsection{}
A similar definition to that in the previous subsection can be given for $\WJJW$. To work with these sets, introduce the following notation.
A representative $(w_1,\ldots,w_d)$ of an element of $\WJJW$ is $\emph{i-stuffed}$ if
\be w_1\in W^{J_1}_1,\ldots,w_{i-1}\in W_{i-1}^{J_{i-1}},\ w_i\in W_i,\ w_{i+1}\in \leftexp{J_i}{W_{i+1}},\ \ldots,\leftexp{J_{d-1}}{W_d}.\ee
It is convenient to represent the element $\mathbf{z}\in \WJJW$, somewhat redundantly, in \emph{stuffed} notation: $\mathbf{z} = (z_1, z_2, \ldots,z_d)$, where $z_i$ is the $i$-th component of the $i$-stuffed expression for $\mathbf{z}$.

\subsection{}
\label{ss TTTC}

The main ideas in this subsection also appear in \cite[\textsection 4]{G2} where they are used to adapt Lusztig's $a$-invariant to give results about the partial order on the cells of $\Res_J \Gamma_W$.

For any $X \subseteq W_1 \times W_2$, define the shorthands
\be
\begin{array}{c}
TT(X) := \left\{T_{w_1}\bt T_{w_2} : (w_1, w_2) \in X \right\},\\
TC(X) := \left\{T_{w_1}\bt \C_{w_2} : (w_1, w_2) \in X \right\},\\
CT(X) := \left\{\C_{w_1}\bt T_{w_2} : (w_1, w_2) \in X \right\}.
\end{array}
\ee
The construction from \textsection \ref{ss IC basis} applied to $\Gamma_{W_2}$ gives the IC basis $\ic_{\tE}(TC({W_1^J} \times W_2))$ of $\tE$.  The next proposition shows that the same canonical basis can be constructed from two other standard bases, and this will be used implicitly in what follows.

\begin{proposition} \label{TCTTprop}
The standard bases

$$ TC({W_1^J} \times W_2),TT({W_1^J} \times W_2) = TT(W_1 \times \JW ), CT(W_1 \times \JW)$$
of $\tE=\HJH$ have the same $A^{-}$-span, denoted $\L$. Moreover,
$$ T_{w_1}\bt \C_{vw_2} \equiv T_{w_1}\bt T_{vw_2} =  T_{w_1v}\bt T_{w_2} \equiv \C_{w_1v}\bt T_{w_2}\mod \ui\L$$
for every $w_1\in {W_1}^J, v\in W_J, w_2\in \JW$. Therefore, the corresponding IC bases are identical:
$$ \ic_{\tE}(TC({W_1^J} \times W_2))=\ic_{\tE}(TT({W_1^J} \times W_2))=\ic_{\tE}(CT(W_1 \times \JW)) $$
(and these will be denoted $\Lambda=\{\C_{w_1,w_2} : (w_1,w_2)\in \WJW \}$).
\end{proposition}
\begin{proof}
The lattices $A^{-}\{T_{w_2}: w_2 \in W_2\}$ and $A^{-}\{\C_{w_2}: w_2 \in W_2\}$ are equal by the definition of an IC basis (\textsection \ref{ss IC basis}).  Thus $A^{-} TC(W_1^J \times W_2) = A^{-} \  TT(W_1^J \times W_2)$ and similarly $A^{-} \   TT(W_1 \times \JW ) = A^{-} \ CT(W_1 \times \JW).$ The remaining statements are clear.

\end{proof}

Now given any lower order ideal $I$ in $\JW$, define $D_{I} = A \ CT(W_1 \times I)$, thought of as an $\H_1$-submodule of $\tE$.  Applying Proposition \ref{p ic lower ideal} to $D_{I} \subseteq \tE$ with poset $W_1 \times \JW$ and lower ideal $W_1 \times D_{I}$ shows that $D_I$ has canonical basis  $\left\{\C_{w_1,w_2} :\ w_1\in W_1, w_2\in I\right\}$ (Proposition \ref{TCTTprop} is used implicitly).  The next theorem now comes easily.

Let $D_{\leq x} = D_{\{w \in \JW: w \leq x\}}$ and $D_{< x} = D_{\{w \in \JW: w < x\}}$.  Recall that $\Gamma_{W_1}$ is the usual $W_1$-graph of the regular representation of $\H_1$.

\begin{theorem}
\label{t HJHcells}
The module $\tE$ (with $W_1$-graph structure $\Lambda$) has a filtration with cellular subquotients that are isomorphic as $W_1$-graphs to $\Gamma_{W_1}$.  In particular, the left cells of $\Lambda$ are isomorphic to those occurring in $\Gamma_{W_1}$.
\end{theorem}
\begin{proof}
For any $x\in \JW$, the map $\pi:D_{\leq x}\to \H_1$ given by $\pi(D_{< x}) = 0$ and $\C_{w}\bt T_x\mapsto\C_{w}$
is an $\H_1$-module homomorphism. Hence the exact sequence
\be \xymatrix{0 \ar[r] & D_{< x} \ar[r] & D_{\leq x} \ar[r]_{\pi}& \H_1 \ar[r]& 0}.\ee
Moreover, $\pi(\tB_{w,x}) = \C_{w}$, which is clear when viewing the $\tB_{w,x}$ as being constructed from the standard basis $CT(W_1\times\JW)$. This gives an isomorphism of $W_1$-graphs $D_{\leq x} / D_{<x} \cong \H_1$.
\end{proof}

Letting $\H$ be the Hecke algebra of $(W, S)$ and setting $\H_1= \H_J$, $\H_2=\H$, $J\subseteq S$, we obtain
\begin{corollary}
The left cells of $\Res_J\H$ are isomorphic as $W_J$-graphs to the left cells of the regular representation of $\H_J$.
\end{corollary}
This corollary is implied by results from \cite[\textsection 5]{HY2}, but the method of proof is different. It is also a consequence of \cite[Theorem 5.2]{R}.
%

By the same methods we can check that the canonical basis construction for induced modules is well-behaved for nested parabolic subgroups.

\begin{proposition}\label{p nested parabolic}
Let $\H$ be the Hecke algebra of $(W,S)$, $J_2 \subseteq J_1 \subseteq S$, $E$ a left $\H_{J_2}$-module with involution $\br{\cdot}$ intertwining that of $\H_{J_2}$, and $\Gamma$ a $\br{\cdot}$-invariant basis of $E$ (like the setup in \textsection \ref{ss IC basis}). Let $\Lambda_{J_1}=\ic_{\H_{J_1}\tsr E}(\{\T_{w,\gamma}:w\in W_{J_1}^{J_2},\gamma\in\Gamma\})$. Then, putting $\tE = \H_{J_2} \tsr E$, we have
\be \label{e nested parab}\ic_{\tE}(\{T_w\bt \gamma:w\in W^{J_2},\gamma\in\Gamma\})=\ic_{\tE}(\{T_w\bt \delta:w\in W^{J_1},\delta\in\Lambda_{J_1}\}). \ee
\end{proposition}
\begin{proof}
By the same argument as in Proposition \ref{TCTTprop}, the right-hand side of (\ref{e nested parab}) can also be constructed from the standard basis
$\{T_{v_1}\bt T_{v_2}\bt \gamma:v_1\in W^{J_1},v_2\in W_{J_1}^{J_2},\gamma\in\Gamma\}$. It remains to check that $W^{J_1} \times W_{J_1}^{J_2} = W^{J_2}$ by $(v_1,v_2) \mapsto v_1 v_2$.  As $v_1$ ranges over left coset representatives of $W_{J_1}$ and $v_2$ over left coset representatives of $W_{J_2}$ inside $W_{J_1}$, $v_1 v_2$ ranges over left coset representatives of $W_{J_2}$ in $W$ (true for any pair of nested subgroups in a group). To see that $v_1 v_2$ is a minimal coset representative, let $x \in W_{J_2}$; then $v_2 \cdot x$ is a reduced factorization and $v_2 x \in W_{J_1}$ (and $v_1$ minimal in $v_1 W_{J_1}$) implies $v_1 \cdot v_2 x$ is a reduced factorization and thus so is $v_1 \cdot v_2 \cdot x$.
\end{proof}

\subsection{}
\label{ss locallabels}
 \newcommand{\ce}{\ensuremath{\Upsilon}}
The set of cells of a $W$-graph $\Gamma$ is denoted $\Cell(\Gamma)$.
We will describe the cells of $\HJJH$ using the results of the previous subsection \textsection \ref{ss TTTC}.

Let $\ce$ be a cell of $\HJH$.  By Theorem \ref{t HJHcells} and its proof, $\ce = \{\C_{w_1,x_2}: w_1 \in \ce'\}$ for some cell $\ce'$ of $\Gamma_{W_1}$ and $x_2 \in \JW$.  We say that $\ce'$ is the \emph{local label} of $\ce$.  By Theorem \ref{t HJHcells}, the cells $\ce$ and $\ce'$ are isomorphic as $W_1$-graphs so that the isomorphism type of a cell is determined by its local label. Thus $\Cell(\HJH)$ has a description via the bijection $\Cell(\HJH) \cong \Cell(\H_1)\times \JW$, $\ce \mapsto (\ce',x_2)$, taking a cell to its local label and an element of $\JW$.  Unfortunately, from this description it is difficult to determine the cells of a cellular subquotient $\H_1\tsr_{J} A\Gamma$ of $\HJH$ for some $\Gamma \in \Cell(H_2)$ (this is a cellular subquotient of $\HJH$ by Proposition \ref{CBSubquotientprop}).

 Essentially the same argument used in Theorem \ref{t HJHcells} yields a similar expression for the general case:
\be \label{e bad cell label} \Cell(\HJJH) \cong \Cell(\H_d)\times \leftexp{J_{1}}W_2\times\ \ldots \times\ \leftexp{J_{d-1}}W_d,\ee taking a cell to its local label and a tuple of right coset representatives.  This of course has the same drawback of it being difficult to identify the subset of cells obtained by taking a cellular subquotient of $\H_d$.  We now address this deficiency.

Put $\tE^k = \H_{d-k}\tsr_{J_{d-k}}\ldots\tsr_{J_{d-1}}\H_{d}$. The collection of cells $\coprod_{k=0}^{d-1} \Cell(\tE^k)$ can be pictured as vertices of an acyclic graph $G$ (see Figure \ref{f VVe+} of \textsection \ref{ss reduced non-reduced}). The subset $\Cell(\tE^k)$ of vertices is the \emph{kth level} of $G$. There is an edge between $\ce^k$ of level $k$ and $\ce^{k+1}$ of level $k+1$ if $\ce^{k+1}\in \Cell(\H_{d-(k+1)}\tsr_{J_{d-(k+1)}}\ce^k)$. Here we are using Proposition \ref{CBSubquotientprop} to identify $\H_{d-(k+1)}\tsr_{J_{d-(k+1)}}\ce^k$ with a cellular subquotient of $\tE^{k+1}$.  Note that from a vertex of level $k+1$ there is a unique edge to a vertex of level $k$ since the cells of a module $\tE^k$ are the composition factors of a composition series for $\tE^k$, thereby yielding composition factors for the induced module of $\tE^{k+1} = \H_{d-(k+1)}\tsr_{J_{d-(k+1)}}\tE^k$.


A vertex $\ce^k$ in the $k$-th level of $G$ has a unique path to a vertex $\ce^0$ in the $0$-th level. The local labels $(\Gamma^k,\ldots,\Gamma^0)$ of the vertices in this path is the \emph{local sequence} of $\ce^k$ (where $\Gamma^i$ is the local label of the vertex in the $i$-th level).

The cell of $\tE^{d-1}$ containing $\tB_{\mathbf{z}}, \ \mathbf{z}\in\WJJW$ is the end of a path with local labels $(\Gamma_{1},\ldots,\Gamma_d)$, where $\Gamma_i\in \Cell(\Gamma_{W_{i}})$ is the cell containing $\C_{z_i}$ and $(z_1,\ldots,z_d)$ is stuffed notation for $\mathbf{z}$.

A local sequence $(\Gamma^{d-1},\ldots,\Gamma^0)$ does not in general determine a cell of $\tE^{d-1}$ uniquely.  For instance, the cells of $\H_{J} \tsr_{J} \H$ with $J = \emptyset$ are just single canonical basis elements of $\H$, so a local sequence does not determine a cell unless the cells of $\H$ are of size $1$.  We say that the tuple $(\tE^{d-1}, \ldots, \tE^0)$ is \emph{weakly multiplicity-free} if there is at most one cell of $\tE^{d-1}$ with local sequence $(\Gamma^{d-1},\ldots,\Gamma^0)$ for all $\Gamma^i \in \Cell(\Gamma_{W_{d-i}})$.  Pure induction $(\H \tsr_J \H_J, \H_J)$ is trivially weakly multiplicity-free since the local label of a cell in $\H \tsr_J \H_J = \H$ is the same thing as the cell itself.  It is not hard to see that $(\tE^{d-1}, \ldots, \tE^0)$ is weakly multiplicity-free if and only if the restriction $(\H_{J_i} \tsr_{J_i} \H_{i+1}, \H_{i+1})$ is for all $i$.

We have seen that the restriction $(\H_J \tsr_J \H, \H)$ is not always weakly multiplicity-free, but a natural question is whether it always is for $J$ of size $|S|-1$.  This fails for $W$ of type $B_2$ and $B_3$ for all choices of $J$ (and presumably for $B_n$, $n > 3$).  This failure may only be because cells in type $B$ do not always correspond to irreducible modules, so this question should be investigated in the unequal parameter setting.  We conjecture the following for type $A$.

\begin{conjecture}\label{c weak mult free}
  If \H is the Hecke algebra of $(W,S) = (\S_n, \{s_1,\ldots,s_{n-1}\})$ and $|J| = |S|-1$, then the restriction $(\H_J \tsr_J \H, \H)$ is weakly multiplicity-free.
\end{conjecture}

This conjecture was verified for $n=10$, $J=S\backslash \{s_5\}$ using Magma, and for $n=16$ and a few arbitrary choices of a cell $\Gamma$, we checked that $(\H_J \tsr_J \Gamma, \Gamma)$ is weakly multiplicity-free. Strangely, it does not seem to be amenable to typical RSK, jeu de taquin style combinatorics.  See \textsection\ref{ss restrict cells} for more about the combinatorics involved here.

\section{Tableau combinatorics}
\label{s tableau combinatorics}
\subsection{}
We will make the description of cells from the previous section combinatorially explicit in the case $W = \S_n$. In this section fix $S = \{s_1,\ldots,s_{n-1}\}$ and $\H$ the Hecke algebra of type $A_{n-1}$. As is customary, we will think of an element of $\S_n$ as a word of length $n$ in the numbers $1,\ldots,n$. We want to maintain the convention used thus far of looking only at left $\H$-modules, however tableau combinatorics is  a little nicer if a right action is used.  To get around this, define the word associated to an element $w = s_{i_1} s_{i_2} \dots s_{i_k} \in W$ to be $w^{-1}(1) \ w^{-1}(2) \dots w^{-1}(n)$, where (to be completely explicit) $w^{-1}(i) = s_{i_k} s_{i_{k-1}} \dots s_{i_1}(i)$ and $s_j$ transposes $j$ and $j+1$. The left descent set of $w\in \S_n$ is $\{s_i :\ w^{-1}(i) > w^{-1}(i+1) \}$.

The RSK algorithm gives a bijection between $\S_n$ and pairs of standard Young tableau (SYT) of the same shape sending $w\in \S_n$ to the pair $(P(w),Q(w))$, written $w\xrightarrow{RSK} (P(w),Q(w))$, where $P(w)$ and $Q(w)$ are the insertion and recording tableaux of the word of $w$ (equal to $w^{-1}(1)\ w^{-1}(2)\dots w^{-1}(n)$ by our convention). As was shown in \cite{KL}, the left cells of $\H$ are in bijection with the set of SYT and the cell containing $\C_w$ corresponds to the insertion tableau of $w$ under this bijection. The cell containing those $\C_w$ such that $w$ has insertion tableau $P$ is the cell \emph{labeled by} $P$.  Note that the shape of the tableau labeling a cell is the transpose of the usual convention for Specht modules, i.e. the trivial representation is labeled by the tableau of shape $1^n$, sign by the tableau of shape $n$.

For the remainder of this paper let $r \in \{1,\ldots,n-1\}$, $J_r = \{s_1,\ldots,s_{r-1}\}$, $J'_{n-r}  = \{s_{r+1},\ldots,s_{n-1}\}$, and $J = J_r \cup J'_{n-r}$.

\subsection{}
\label{ss induced cells}
Let $\Gamma$ be a cell of $W_J$ labeled by a pair of insertion tableaux $(T,T') \in \Tab_{1^r0^{n-r}}\times\Tab_{0^r1^{n-r}}$, where $\Tab_\alpha$ is the set of tableau with $\alpha_i$ entries equal to $i$. Here we are using the easy fact, proven carefully in \cite{R}, that a cell of $\Gamma_{W_1\times W_2}$ is  the same as a cell of $\Gamma_{W_1}$ and one of $\Gamma_{W_2}$.  We will describe the cells of $\H \tsr_J A\Gamma$.

For any $w\in W$, in the notation of \textsection \ref{ss coxeter group}, $\rj{w}{J} = (a, b)\in W_{J_r}\times W_{J'_{n-r}}$, where $a$ (resp. $b$) is the permutation of numbers $1,\ldots,r$ (resp. $r+1, \ldots,n$) obtained by taking the subsequence of the word of $w$ consisting of those numbers. For example, if $n=6$, $w = 436125$, and $r = 3$, then $a = 312$ and $b = 465$.

The induced module $\H\tsr_{J}A\Gamma$ has canonical basis $\{\C_w: P(\rj{w}{J}) = (T,T')\}$, where we define $P(a, b)$ for $(a,b)\in W_{J_r}\times W_{J'_{n-r}}$ to be $(P(a),P(b))$.
For any tableau $P$, let $\jdt(P)$ denote the unique straight-shape tableau in the jeu de taquin equivalence class of $P$. From the most basic properties of insertion and jeu de taquin it follows that if $\rj{w}{J} = (a, b)$, then $P(w)_{\leq r} = P(a),\ P(w)_{> r} = \jdt(P(b))$, where $P_{\leq r}$ (resp. $P_{> r}$) is the (skew) subtableau of $P$ with entries $1,\ldots,r$ (resp. $r+1,\ldots,n$).   See, for instance, \cite[A1.2]{St} for more on this combinatorics.  We now have the following description of cells.
\begin{proposition}
\label{p cells of induce}
With $\Gamma$ labeled by $T,T'$ as above, the cells of $\H\tsr_{J}A\Gamma \subseteq \H$ are those labeled by  $P$ such that $P_{\leq r} = T$, $\jdt(P_{> r}) = T'$.
\end{proposition}
\begin{example}
Let $n=6$, $r=3$,  and $T,T'=\left({\tiny\young(12,3),\young(46,5)}\right)$. Then the cells of $\H\tsr_{J}A\Gamma$ are labeled by
$$ {\tiny\young(1246,35)},\ {\tiny\young(124,356)},\ {\tiny\young(1246,3,5)},\ {\tiny\young(126,34,5)},\ {\tiny\young(124,36,5)},\ {\tiny\young(12,34,56)},\ {\tiny\young(126,3,4,5)},\ {\tiny\young(12,36,4,5)}.$$
\end{example}

This is, of course, the Littlewood-Richardson rule. The combinatorics of the Littlewood-Richardson rule matches beautifully with the machinery of canonical bases.   This version of the Littlewood-Richardson is due to Sch\"utzenberger and its connection with canonical bases was also shown in \cite{R}.

Let $V_{\lambda}$ be the Specht module corresponding to the partition $\lambda$, and put $\mu = \sh(T),\ \nu=\sh(T')$.
It was established in \cite{KL} that all left cells of $\H$ isomorphic at $\u=1$ to $V_{\lambda}$ are isomorphic as $W$-graphs.  This, together with the fact that the $W$-graph of Theorem \ref{t HY} depends only on the isomorphism type of the $W_J$-graph $\Gamma$, shows that the multiplicity of $V_{\lambda}$ in $\ind^W_{W_J}
(V_{\mu}\boxtimes V_{\nu})$ is given by the combinatorics above and is independent of the chosen insertion tableaux $T, T'$.

\subsection{}
\label{ss restrict cells}
Let $\Gamma$ be a cell of $\H$ labeled by $P$  with $\sh(P) = \lambda$.  We will describe the cells of $\Res_J A \Gamma$.

For any $w\in W$, $\lj{w}{J} = (a, b)\in W_{J_r}\times W_{J'_{n-r}}$, where $a$ (resp. $b$) is the permutation of numbers $1,\ldots,r$ (resp. $r+1, \ldots,n$) with the same relative order as $w^{-1}(1) \ w^{-1}(2) \dots w^{-1}(r)$ (resp. $w^{-1}(r+1) \dots w^{-1}(n)$).  For example, if $n=6$, $w = 436125$, and $r = 3$, then $a = 213$ and $b = 456$.

Specifying a cell $\ce$ of $\Res_J \H$ is equivalent to giving $x \in \leftexp{J}{W}$ and  $(T,T') \in \Tab_{1^r0^{n-r}}\times\Tab_{0^r1^{n-r}}$.    Under this correspondence, $\ce = \{\C_w: P(\lj{w}{J}) = (T,T'), \rJ{w}{J} = x\}$.

 Given $\mu\vdash r,\nu\vdash n-r$, define
 \be \mu\sqcup\nu =(\nu_1+\mu_1,\nu_2+\mu_1,\ldots,\nu_{\ell(\nu)}+\mu_1,\mu_1,\mu_2,\ldots,\mu_{\ell(\mu)}), \ee
  where $\ell(\mu)$ is the number of parts of $\mu$.

Expressing the tableaux on $1,\ldots,r$ and $r+1,\ldots,n$ that label the cells of $\Res_JA\Gamma$ in terms of $P$ is tricky: first define the set
\be X := \{(T,T') :\ |T| = r, |T'| = n-r, \jdt(T T') = P\},\ee
where $T T'$ is the tableau of shape $\mu \sqcup\nu / \rho$ ($\sh(T)=\mu$, $\sh(T')=\nu$, $\rho = {\mu_1}^{\ell(\nu)}$) obtained  by adding $T'$ to  the top right of $T$.
The multiset of local labels of the cells of $\Res_JA\Gamma$ (Conjecture \ref{c weak mult free} says this is actually a set) is obtained by projecting each element of $X$ onto the set $\Tab_{1^r0^{n-r}}\times\Tab_{0^r1^{n-r}}$ by replacing the entries of $T$ (resp. $T'$) by $1,\ldots,r$ (resp. $r+1,\ldots,n$) so that relative order is preserved.

\begin{example}
If $n=6$, $r=3$, and $P={\tiny\young(125,36,4)}$, then $X$ is
$$\left\{ \left({\tiny\young(36,4),\young(125)}\right), \left({\tiny\young(1,3,4),\young(25,6)}\right), \left({\tiny\young(16,4),\young(25,3)}\right), \left({\tiny\young(146),\young(25,3)}\right), \left({\tiny\young(13,4),\young(25,6)}\right), \left({\tiny\young(13,4),\young(2,5,6)}\right) \right\}.$$
Hence the cells of $\Res_JA\Gamma$  have local labels
$$\left({\tiny\young(13,2),\young(456)}\right), \left({\tiny\young(1,2,3),\young(45,6)}\right), \left({\tiny\young(13,2),\young(46,5)}\right), \left({\tiny\young(123),\young(46,5)}\right), \left({\tiny\young(12,3),\young(45,6)}\right), \left({\tiny\young(12,3),\young(4,5,6)}\right).$$
\end{example}

A slightly better description of the cells of $\Res_J A\Gamma$  is as follows.  Fix $\mu \vdash r$, $\nu \vdash n-r$ such that $\lambda \subseteq \mu \sqcup \nu$, and $B$ a tableau of the rectangle shape $\rho := {\mu_1}^{\ell(\nu)}$. Now consider the jeu de taquin growth diagrams with lower left row corresponding to $P$, lower right row corresponding to $B$, and the partition at the top equal to $\mu \sqcup \nu$ (see, e.g., \cite[A1.2]{St}).   The upper right row of such a growth diagram necessarily corresponds to some $T T'$ such that $\jdt(TT') = P$, and the upper left row corresponds to some $A$ such that $\jdt(A) = B$.  Since a growth diagram is constructed uniquely from either of its sides, we obtain the bijection
\be \left\{(T,T') :\ \sh(T)=\mu,\sh(T')=\nu,\jdt(T T') = P\right\} \cong \{A :\ \sh(A) = \mu\sqcup\nu/\lambda,\jdt(A) = B\}.\ee
>From an $A$ in the set above, one obtains the corresponding $(T,T')$ as follows: perform jeu de taquin to $P$ in the order specified by the entries of $A$ to obtain a tableau of shape $\mu\sqcup\nu /\rho$; split this into a tableau of shape $\mu$ and one of shape $\nu$. This can be used to give another description of the set $X$.  This description has the advantage that the same choice of $B$ can be used for all tableau $P$  of shape $\lambda$.

\subsection{}
\label{ss sb}
If $r=1$ or $r = n-1$, then restricting and inducing are multiplicity-free. Therefore, we only need to keep track of the shapes of the tableaux rather than the tableaux themselves, except at the first step $\Cell(\H_d)$, in order to determine a cell of $\HJJH$.   However, it is often convenient for working concrete examples to keep track of all tableaux.

If $r=1$ or $r=n-1$, then
the cells of $\Res_JA\Gamma$, with $\Gamma$ labeled by $P$, can be described explicitly.  If $r=1$ (resp. $r=n-1$), they are labeled by the tableaux obtained from $P$ by column-uninserting (resp. row-uninserting) an outer corner
and replacing the entries of the result with $2,\ldots,n$ (resp. $1,\ldots,n-1$) so that relative order is preserved.

 We will work with both $r=1$ and $r=n-1$ in this paper because tableau combinatorics is easier with $r=n-1$, but $r =1$  is preferable for our work in  \textsection\ref{s tensor V} and beyond.  It is therefore convenient to be able to go back and forth between these two conventions.

On the level of algebras, this is done by replacing any $\H_K$-module by the $\H_{w_0 K w_0}$-module obtained by twisting by the isomorphism $\H_{w_0 K w_0} \cong \H_K, T_{s_i} \mapsto T_{s_{n-i}}$, where $w_0$ is the longest element of $W$.
Combinatorially, this corresponds to replacing a word $x_1 x_2 \dots x_n$ with $x^\sharp := n+1-x_1\ n+1-x_2 \dots n+1-x_n$.  The local label of a cell changes from $T$ to $\evac(T)$, where $T \mapsto \evac(T)$ is the Sch\"utzenberger involution (see, e.g., \cite[A1.2]{St}).  More precisely, the local label $(T,T') \in \Tab_{1^j0^{n-j}}\times\Tab_{0^j1^{n-j}}$ of a cell of an $\H_{S\backslash s_j}$-module becomes $(\evac(T')^*,\evac(T)^*)$, where $\evac(T')^*$ (resp. $\evac(T)^*$) is obtained from $\evac(T')$ by adding a constant to all entries so that $\evac(T')^* \in \Tab_{1^{n-j}0^{j}}$ (resp.  $\evac(T)^* \in \Tab_{0^{n-j}1^{j}}$).

\subsection{}
\label{ss affinecells}
In this subsection we  give a combinatorial description of cells of a certain submodule of $\Res_\H\eH$, where $\eH$ is the extended affine Hecke algebra of type $A$. We digress to introduce this object. See \cite{X}, \cite{H}  for a more thorough introduction.

First of all, everything we have done so far for Coxeter groups also holds for extended Coxeter groups. An extended Coxeter group, defined from a Coxeter group $(W, S)$ and an abelian group $\Pi$ acting by automorphisms on $(W, S)$, is the semi-direct product $\Pi \ltimes W$, denoted $\eW$. The length function and partial order on $W$ extend to $\eW$: $\ell(\pi v) = \ell(v)$, and $\pi v \leq \pi' v'$ if and only if $\pi = \pi'$ and $v \leq v'$, where $\pi, \pi' \in \Pi$, $v, v' \in W$. The definitions of left and right descent sets, reduced factorization, the $\br{\cdot}$-involution, and definition of the Hecke algebra (\ref{e Hecke algebra def}) of \textsection\ref{s Hecke algebra} carry over identically. The Hecke algebra elements $T_\pi$ for $\pi \in \Pi$ will be denoted simply by $\pi$; note that these are $\br{\cdot}$-invariant.

Although it is possible to allow parabolic subgroups to be extended Coxeter groups, we define a parabolic subgroup of $\eW$ to be an ordinary parabolic subgroup of $W$ to simplify the discussion (this is the only case we will need later in the paper).  With this convention, each coset of a parabolic subgroup $\eW_J$ contains a unique element of minimal length.

In the generality of extended Coxeter groups, a $\eW$-graph $\Gamma$ must satisfy $\pi\gamma \in \Gamma$ for all $\pi \in \Pi$, $\gamma \in \Gamma$ in addition to (\ref{Wgrapheq}). The machinery of IC bases carries over without change. Everything we have done so far holds in this setting; the only thing that needs some comment is Theorem \ref{t HY}. Presumably the proof carries over without change, however it is also easy to deduce this from Theorem \ref{t HY} for ordinary Coxeter groups: use the fact that $\tP_{\pi x, \delta, \pi v, \gamma} = \tP_{x, \delta, v, \gamma}$ to deduce that with the definition (\ref{e mudef}) for $\mu$, $\mu(\pi x, \delta, \pi v, \gamma) = \mu(x, \delta, v, \gamma)$ ($x, v \in W, \pi \in \Pi$); the identity $\tB_{\pi v, \gamma} = \pi \tB_{v, \gamma}$ together with the theorem for ordinary Coxeter groups give it for extended Coxeter groups.

Let $W, \aW$ be the Weyl groups of type $A_{n-1}, \tilde{A}_{n-1}$ respectively. Put $K_j = \{s_0, s_1, \ldots,\hat{s}_j,\ldots, s_{n-1}\}$.  Let $Y\cong \ZZ^n,\ Q\cong \ZZ^{n-1}$ be the weight lattice, root lattice of $GL_n$. The extended affine Weyl group $\eW$ is both $Y \rtimes W$ and $\Pi \ltimes \aW$ where $\Pi \cong Y / Q \cong \ZZ$. For $\lambda \in Y$, let $y^\lambda$ be the corresponding element of $\eW$ and let $y_i = y^{\epsilon_i}$, where $\epsilon_1, \ldots, \epsilon_n$ is the standard basis of $Y$. Also let $\pi$ be the generator of $\Pi$ such that $s_i\pi = \pi s_{i-1}$, where subscripts are taken mod $n$. The isomorphism $Y \rtimes W \cong \Pi \ltimes \aW$ is determined by
\be \label{e Y W=Pi Wa} y_i \to s_{i-1} \ldots s_1 \pi s_{n-1} \ldots s_i, \ee
and the condition that $W \hookrightarrow Y \rtimes W \cong \Pi \ltimes \aW$ identifies $W$ with $\aW_{K_0}$ via $s_i \mapsto s_i$, $i \in [n]$.

Another description of $\eW$, due to Lusztig, is as follows. The group $\eW$ can be identified with the group of permutations $w: \ZZ \to \ZZ$ satisfying $w(i+n) = w(i)+n$ and $\sum_{i = 1}^n (w(i) - i) \equiv 0 \mod n$. The identification takes $s_i$ to the permutation transposing $i+kn$ and $i+1+kn$ for all $k \in \ZZ$, and takes $\pi$ to the permutation $k \mapsto k+1$ for all $k \in \ZZ$. We can then express an element $w$ of $\eW$ in \emph{window notation} as the sequence of numbers $w^{-1}(1) \dots w^{-1}(n)$, also referred to as just the word of $w$. For example, if $n = 4$ and $w = \pi^2 s_2 s_0 s_1$, then the word of $w$ is $\text{-}3203$.

Let $\pY = \ZZ^n_{\geq 0}$ and $\pW = \pY \rtimes W$. There is a corresponding subalgebra $\pH$ of $\eH$, equal to both $A \{ T_w : w \in \pW \}$ and $A \{ \C_w : w \in \pW \}$ \cite{B}. Let $\Gamma$ be a $W$-graph and put $E = A\Gamma$. The \emph{positive, degree $d$} part of $\Res_\H \eH \tsr_\H E$ is
\be (\pH \tsr_\H E)_d := A\{\tB_{y^\lambda v,\gamma}:\lambda \in \pY, |\lambda| = d, v \in W \text{ such that } y^\lambda v \in \eW^{K_0}, \gamma\in\Gamma\}. \ee

\begin{proposition}
$(\pH \tsr_\H E)_d$ is a cellular submodule of $\Res_\H \eH \tsr_\H E$.
\end{proposition}
\begin{proof}
The $A$-basis above can be rewritten as $\{ \pi^d \tB_{w,\gamma} : \pi^d w \in \pW, w \in \aW, \gamma \in \Gamma \}$. It is easy to see this is left stable by the action of $\H$, given that $\pH$ is a subalgebra of $\eH$ containing $\H$.
\end{proof}

Let $\Gamma$ be a cell of $\H$ labeled by $T$ and $\hE^1= (\pH \tsr_{\H} A \Gamma)_1$.  We now return to give a combinatorial description of the cells of $\hE^1$. The restriction $(\Res_\H \eH \tsr_\H E, \eH \tsr_\H E)$ is not weakly multiplicity-free, so we have to use the description (\ref{e bad cell label}).  In this case, we have found it most convenient to use a hybrid of the description in (\ref{e bad cell label}) and local labels, which we now describe.

Given $x \in \eW$, define $P(x)$ to be the insertion tableau of the word of $x$.  Since $\lj{x}{K_0}$ is the permutation of $1, \ldots, n$ with the same relative order as the word of $x$, $P(\lj{x}{K_0})$ is obtained from $P(x)$  by replacing the entries with $1,\ldots,n$ and keeping relative order the same.

Let $a_k = s_{k-1}\ldots s_{1}$ for $k\in \{2,\ldots,n\}$, $a_1 = 1$ be the minimal left coset representatives of $W_{J'_{n-1}}$.  Then $\hE^1 = A \{\tB_{a_k \pi, w}: k \in [n], P(w) = T \}$.  In this case, define the local label of the cell containing $\tB_{a_k \pi, w}$ to be $P(a_k \pi w)$.  A caveat to this is that if we then form some induced module $\H_1 \tsr_{J_1} \hE^1$, it is good to convert the local labels of $\hE^1$ to be the tableaux $P(\lj{(a_k \pi w)}{K_0})$ before computing local labels of  $\H_1 \tsr_{J_1} \hE^1$ (see Figure \ref{f VVe+ affine} of \textsection\ref{ss reduced non-reduced}).

Combinatorially, the cells of $\hE^1$ may be described as follows.  Let $w \in W$ with $P(w) = T$ and define $Q = Q(w)$. Let $\lj{w}{J_{n-1}}^*$ be the word obtained from $w$ by deleting its last number (see Example \ref{ex affine insert}). Then $\lj{w}{J_{n-1}}^*\xrightarrow{RSK}(T^-,Q_{\leq n-1})$, where $T^-$ is obtained from $T$ by uninserting the square $Q\backslash Q_{\leq n-1}$; let $c$ be the number uninserted.  Write $a_k\pi w$ in window notation, which is $\lj{w}{J_{n-1}}^*$ with a $c-n$ inserted in the $k$-th spot. Let $Q^+$ be the tableau obtained by column-inserting $k$ into the tableau obtained from $Q_{\leq n-1}$ by replacing entries with $\{1,\ldots,k-1,k+1,\ldots,n\}$ and keeping the same relative order. We have $a_k\pi w\xrightarrow{RSK}(T^+,Q^+)$, where $T^+$ is $\jdt(T^-,Q^+\backslash Q_{\leq n-1})$ with the number $c-n$ added to the top left corner (so that the resulting tableau has a straight-shape).
This implies the following result about the cells of $\hE^1$.

\begin{proposition}
The local labels of the cells of $\hE^1$ are those tableaux obtained from $T$ by uninserting some outer corner then performing jeu de taquin to some inner corner, and finally filling in the missing box in the top left with a $c-n$, where $c$ is the entry bumped out in the uninsertion.
\end{proposition}

\begin{example}
\label{ex affine insert}
For the element $(a_3 \pi, 346512) \in \eW^{K_0} \times W$, the insertion and recording tableaux discussed above are
\hoogte=10pt
\breedte=9pt
\dikte=0.3pt
\be \begin{array}{cccc}
& a_3\pi w & \lj{w}{J_{n-1}}^* & w\\
& 34\ng 4651 & 34651 & 346512\\
P &
{\tiny\young(\mfour 15,34,6)}
 & {\tiny\young(145,3,6)} & {\tiny\young(125,34,6)}\\
 \ \\
Q & {\tiny\young(124,35,6)} & {\tiny\young(123,4,5)} & {\tiny\young(123,46,5)}
\end{array} \ee
\end{example}

\section{Computations of some $\C_w$}
\label{s computation c_w}
Suppose in what follows that $r=n-1$. Let $b_k := s_k\ldots s_{n-1}$ for $k\in [n-1]$ and $b_n = 1$ be the elements of $W^J$. It is possible to write down explicitly an element from each cell of $\H\tsr_{J}\H$ in terms of the canonical basis of $\H$. This is the content of the following theorem, which we include mainly for its application in the next section.  It is quite interesting for its own sake, however, given that it does not seem to be known how to write down an element from each cell of $\H$ in terms of the $T$'s.

\begin{proposition}\label{p easyCw}
Let $\Gamma$ be a $W_J$-graph, and $\gamma\in\Gamma$ satisfying $K:=\{s_k,\ldots,s_{n-2}\}\subseteq L(\gamma)$. Then
$$ \tB_{b_k,\gamma}=\frac{1}{[n-k]!}\C_{b_k\lj{w_0}{K}}\bt\gamma = \left(T_{b_{k}} + \ui T_{b_{k+1}} + \ldots + \u^{k-n}T_{b_n}\right)\bt\gamma. $$
\end{proposition}

\begin{proof}
The right hand equality follows from $K\subseteq L(\gamma)$ and the well-known identity $\C_{b_k\lj{w_0}{K}} = \C_{\lj{w_0}{K \cup s_{n-1}}} = \left(T_{b_k} + \ui T_{b_{k+1}} + \ldots + \u^{k-n}T_{b_n}\right)\C_{\lj{w_0}{K}}$ ($\lj{w_0}{K}$ is the longest element of $W_K$). Once this is known we have produced an element that is both $\br{\cdot}$-invariant (being equal to $\frac{1}{[k-1]!}\C_{b_k\lj{w_0}{K}}\bt\gamma$) and congruent to $T_{b_k}\bt\gamma\mod\ui\L$ (being equal to $\left(T_{b_{k}} + \ui T_{b_{k+1}} + \ldots + \u^{k-n}T_{b_n}\right)\bt\gamma$).
 \end{proof}

\begin{theorem}\label{specialelementstheorem}
Let $\ce$ be the cell of $\H\tsr_{J}\H$ determined by $\lambda^{(1)},\mu,P$, where $\lambda^{(1)},\mu,\sh(P)$ are partitions of $n,n-1,$ and $n$ respectively satisfying $\mu \subseteq \lambda^{(1)},\sh(P)$. Then $\ce$ contains an element
$$ \tB_{b_{k'},w} = \left(T_{b_{k'}} + \ui T_{b_{k'+1}} + \ldots + \u^{-k+1}T_{b_n}\right)\bt \C_w,$$
where the $k$-th row of $\lambda^{(1)}$ contains the square $\lambda^{(1)}/\mu$, $k' = n + 1 - k$, and $w$ satisfies $\{s_{k'},\ldots,s_{n-2}\}\subseteq L(w)$.
\end{theorem}

\begin{proof}
To construct a desired $w$, let $Q$ be any tableau of shape $\lambda^{(2)}$ such that $Q_{< n}$ has a $k'-1+r$ in the last box of the $r$-th row for $r \in \{1,\ldots,k-1\}$ (see Example \ref{e special element}) and $Q_{\geq n}$ is the square $\lambda^{(2)}/\mu$. Define $w$ by $w\xrightarrow{RSK}(P,Q)$.

Consider the element $(b_{k'},w) = \mathbf{z}\in W \stackrel{J}{\times} W$, which is $(b_{k'}\lj{w}{J},w)$ in stuffed notation. Now $Q(b_{k'}\lj{w}{J}) = P(\lj{w}{J}^{-1}s_{n-1}\dots s_{k'}) = Q_{<n}^* \leftarrow k'$, where $Q_{<n}^*$ is $Q_{<n}$ with $1$ added to all numbers $\geq k'$, and $T \leftarrow a$ denotes the row-insertion of $a$ into $T$. By construction of $Q_{<n}$, the bumping path of inserting $k'$ into $Q_{<n}^*$ consists of the last square in rows $1,\ldots,k$, the last square in the $k$-th row being the newly added square. Therefore, $\tB_{\mathbf{z}}$ is contained in $\ce$ because the shape of $Q_{<n}^* \leftarrow k'$ is $\lambda^{(1)}$.

Remembering our convention for the word of $w$, the left descent set $L(w)$ can be read off from $Q$: it is the set of $s_i$ such that $i+1$ occurs in a row below the row containing $i$. In particular, $K := \{s_{k'},\ldots,s_{n-2}\}\subseteq L(w)$. The theorem follows from Proposition \ref{p easyCw}.
\end{proof}

\begin{example}\label{e special element}
\newcommand{\bfsix}{\ensuremath{\mathbf{6}}}
\newcommand{\bfseven}{\ensuremath{\mathbf{7}}}
\newcommand{\bfeight}{\ensuremath{\mathbf{8}}}

If $n=9$, $k=4$, $\mu = (3,2,2,1)$, and $P = {\tiny\young(158,269,37,4)}$, we could choose
$Q = {\tiny \young(12\bfsix,3\bfseven9,4\bfeight,5)}$ or any tableau with the given numbers in bold. Then,
$$ b_k'\lj{w}{J},\lj{w}{J},w=473219865,47321865,473219658, $$
and $Q_{<n} = {\tiny \young(126,37,48,5)}$, $Q(b_k'\lj{w}{J}) = {\tiny \young(126,37,48,59)}$.
\end{example}

\section{Canonical maps from restricting and inducing}
\label{s canonical maps}

\subsection{}
The functor $\H\tsr_J-:\H_J$-$\Mod \to \H$-$\Mod$ is left adjoint to $\Res_J:\H$-$\Mod\to \H_J$-$\Mod$. Let $\unit$ (resp. $\counit$) denote the unit (resp. counit) of the adjunction so that $\unit(F)\in \hom_{\H_J\text{-}\Mod}(F,\Res_J\H\tsr_JF)$ corresponds to $\id_{\H\tsr_JF}$ (resp. $\counit(E)\in \hom_{\H\text{-}\Mod}(\H\tsr_J\Res_JE,E)$ corresponds to $\id_{\Res_JE}$). The unit (resp. counit) is a natural transformation from the identity functor on $\H_J$-$\Mod$ to the functor $\Res_J \H \tsr_J-$ (resp. from the functor $\H \tsr_J \Res_J$ to the identity functor on $\H$-$\Mod$). We will omit the argument $F$ or $E$ in the notation for the unit and counit when there is no confusion. Explicitly, $\unit:F\to\Res_J\H\tsr_JF$ is given by $f\mapsto 1\bt f$, and $\counit:\H\tsr_JE\to E$ is given by $h\bt e\mapsto he$.  It is clear from these formulas that the unit and counit intertwine the  involution $\br{\cdot}$.

\subsection{}
The unit behaves in a simple way on canonical basis elements.
\begin{proposition}\label{p res ind}
Let $F=A\Gamma$ be any $\H_J$-module coming from a $W_J$-graph $\Gamma$. The map $\unit:F \to \Res_J \H\tsr_J F$ takes canonical basis elements to canonical basis elements.  Therefore $\im(\unit)$ is a cellular submodule isomorphic to $A\Gamma$ as a $W_J$-graph.
\end{proposition}
\begin{proof}
The elements $\tB_{1,\gamma} = \unit(\gamma)$ ($\gamma\in \Gamma$) are canonical basis elements and are an $A$-basis for the image of $\unit$.
\end{proof}

\subsection{}
Again, restrict to the case where $W$ and $\H$ are of type $A_{n-1}$, $S = \{s_1,\ldots,s_{n-1}\}$, and $J = S\backslash s_{n-1}$.

We are not able to give a good description of where the counit $\counit$ takes canonical basis elements in general, but we have a partial result along these lines, assuming the following conjecture.

\begin{conjecture}\label{c dominance}
Let $\Lambda$ be the $W_1$-graph on $\HJJH$ with $\H_1$ of type $A$.  If $\mathbf{y} \leq_\Lambda \mathbf{z}$, $\mathbf{y,z} \in \Lambda$ and $\mathbf{y,z}$ are in cells with local labels of shape $\lambda, \mu$ respectively, then $\lambda < \mu$ in dominance order.
\end{conjecture}

Let $\Gamma$ be a cell of $\H$ and $\tau: A\Gamma \to A\Gamma$ an $\H$-module homomorphism.  We want to conclude that $\tau$ is multiplication by some constant $c \in A$.  This can be seen, for instance, by tensoring with $\CC$ over $A$ using any map $A \to \CC$ that does not send $\u$ to a root of unity; Schur's Lemma applies as $\H \tsr_A \CC \cong \CC \S_n$ and $A \Gamma \tsr_A \CC$ is irreducible.   Thus $\tau \tsr_A \CC = a \ \id$, $a \in \CC$ for infinitely many specializations of $\u$ implies $\tau = c \ \id$, $c \in A$.

Let $X_{\lambda}$ be the two-sided cell of $\H$  consisting of the cells labeled by tableaux of shape $\lambda \vdash n$. Let $\Gamma$ be a cell of $\H\tsr_{J}X_{\lambda}$ with local sequence $P_1,P_2$ both of shape $\lambda$. Conjecture \ref{c dominance} implies $A \Gamma$ is a submodule of   $(\H\tsr_{J}X_{\lambda})/ X$, where $X$ is the cellular submodule consisting of those cells of dominance order $< \lambda$.   By a similar argument to the one above, $\counit(X) = 0$.  Therefore the map $\H\tsr_{J}X_{\lambda} \xrightarrow{\counit} X_\lambda$ gives rise to a map $A \Gamma \xrightarrow{\counit} X_\lambda$.

Letting
$\ce$ be a cell of $X_\lambda$, we have
\be \label{e [k]}
A\Gamma \xrightarrow{\counit} X_{\lambda} \xrightarrow{p} A\ce \cong A\Gamma. \ee
The map $p$ is a cellular quotient map by \cite[Corollary 1.9]{L} and the rightmost isomorphism of $W$-graphs comes  from the fact that any two cells with the same local label are isomorphic as $W$-graphs (\textsection\ref{ss locallabels}).   We now can state the main application of Theorem \ref{specialelementstheorem}.

\begin{corollary}
\label{c [k]}
Assuming Conjecture \ref{c dominance} and with the notation above, if the square $P_1 \backslash (P_1)_{<n}$ lies in the $k$-th row, then the composition of the maps in (\ref{e [k]}) is $[k]\ \id$ if $\ce$ is labeled by $P_2$ and 0 otherwise.
\end{corollary}
\begin{proof}
By the discussion above, this composition must be $c \ \id$ for some $c \in A$.
Apply Theorem \ref{specialelementstheorem} with $\lambda^{(1)} = \lambda$, $\mu = \sh((P_1)_{<n})$, $P = P_2$, noting that from the construction of $w$ in the proof, $\{k',\ldots,n-1\}\subseteq L(w)$ in this case. Therefore \be \counit(\tB_{b_{k'},w}) = \left(T_{b_{k'}} + \ui T_{b_{k'+1}} + \ldots + \u^{-k+1}T_{b_n}\right)\C_w = [k]\C_w.\ee
\end{proof}


It is tempting to conjecture that $\counit(\tB_{b_{k'}, w})$ is a constant times a canonical basis element of $\H$, where $(b_{k'},w)$ is as constructed in Theorem \ref{specialelementstheorem}, but this is false in general. The following counterexample was found using Magma.
\begin{example}Let $n=6$, $k=2$, $k'=4$, $w=521634\xrightarrow{\tiny{RSK}}\left({\tiny\young(134,26,5), \ \young(146,25,3)}\right)$. Then,
\be \tB_{b_{k'},
w} = \left(T_{b_4} + \ui T_{b_5} + \u^{-2}\right)\bt \C_{521634} \stackrel{\vphantom{}^\counit}{\longmapsto} [2]\C_{521643} + [2]\C_{321654}.\ee
The element $(b_{k'},w) \in W \stackrel{J}{\times} {W}$ is $(421653,521634)$ in stuffed notation and the cell containing it has local label  ${\tiny\young(13,25,46)}$ . The labels of the cells containing $521643, 321654$ are ${\tiny\young(13,24,56),\young(14,25,36)}$ respectively.
\end{example}

\section{Some $W$-graph versions of tensoring with the defining representation}
\label{s tensor V}

Let $V$ denote the $n$-dimensional defining representation of $\S_n$: $V = \ZZ\{ x_1,\ldots,x_n\}$, $s_i(x_j) = x_{s_i(j)}$. In this section, we will explore three $W$-graph versions of tensoring with $V$.  We then look at $W$-graphs corresponding to tensoring twice with $V$ and show that these decompose into a reduced and non-reduced part. We make a habit of checking what our $W$-graph constructions become at $\u = 1$ in order to keep contact with our intuition for this more familiar case.

\subsection{}

In what follows, $E$ denotes an $\H$-module or $\ZZ \S_n$-module, depending on context. A useful observation, and indeed, what motivated us to study inducing $W$-graphs is that $V\tsr E\cong \ZZ\S_n\tsr_{\ZZ\S_{n-1}}E$ for any $\ZZ\S_n$-module $E$. This is well-known, but the proof is instructive.

\begin{proposition}
\label{p gk}
Given a finite group $G$, a subgroup $K$, and a $\ZZ(G)$-module $E$, there is a ($\ZZ G$-module) isomorphism, natural in $E$
\begin{equation}
\label{e gk}
\begin{array}{ccc}
\ZZ G\tsr_{\ZZ K}E & \cong & \left(\ZZ G\tsr_{\ZZ K}\ZZ\right) \tsr_{\ZZ} E,\\
g\bt e & \to & (g\bt 1)\tsr ge,
\end{array}
\end{equation}
where $\ZZ$ denotes the trivial representation of $K$.
\end{proposition}

\begin{proof}
The expressions $g k\bt k^{-1} e$ and $g \bt e$ ($k \in K$) are sent to the same element so this map is well-defined.   Similarly, its inverse $(g \bt 1) \tsr e \mapsto g\bt g^{-1} e$ is well-defined.  These maps clearly intertwine the action of $G$.
\end{proof}

Maintain the notation $W = \S_n$, $J_{n-1} = \{s_1,\ldots,s_{n-2}\}$, $J'_{n-1} = \{s_2,\ldots,s_{n-1}\}$  of the previous sections. Recall that $b_k = s_{k} \ldots s_{n-1}$ for $k \in [n-1]$, $b_n = 1$ are the minimal left coset representatives of $W_{J_{n-1}}$, and $a_k = s_{k-1}\ldots s_{1}$ for $k\in \{2,\ldots,n\}$, $a_1 = 1$ the minimal left coset representatives of $W_{J'_{n-1}}$.

\begin{corollary}\label{c indresiso}
For the inclusions $\S_{n-1}=W_{J'_{n-1}} \hookrightarrow W = \S_n$ and $\S_{n-1}=W_{J_{n-1}} \hookrightarrow W = \S_n$, we have $ \ZZ\S_n\tsr_{\ZZ\S_{n-1}}E \cong V\tsr_\ZZ E$ for any $\ZZ\S_n$-module $E$.
\end{corollary}

\begin{proof}
Put $G=\S_n$. If $K=W_{J'_{n-1}}$, then $\ZZ G\tsr_{\ZZ K} \ZZ\cong V$ by $a_i\bt 1\mapsto x_i$. If $K=W_{J_{n-1}}$, then $\ZZ G\tsr_{\ZZ K} \ZZ\cong V$ by $b_i\bt 1\mapsto x_i$.
\end{proof}

The Hecke algebra is not a Hopf algebra in any natural way, so it is not clear what a Hecke algebra analogue of $F\tsr E$ should be for $F,E$ $\ZZ\S_n$-modules. If $F= V$, however, then $\H \tsr_J E$ is a $\u$-analogue of $\ZZ\S_n\tsr_{\ZZ\S_{n-1}}E\cong V\tsr E$, where $r$ is either $n-1$ or $1$ (and $J = J_r \cup J'_{n-r}$).  These choices for $r$  give isomorphic representations at $\u=1$, but do not give isomorphic $W$-graphs in general.

\begin{example} \label{ex 2wgraph}

Let $e^{+}$ be the trivial representation of $\H$. Then compare $\H\tsr_{{J'_{n-1}}}e^{+}$ (first row) with $\H\tsr_{{J_{n-1}}}e^{+}$ (second row) for $n=4$ :
$$
\xymatrix{
{\tB_{a_4,e^{+}}}^{1,2,3} & {\tB_{a_3,e^{+}}}^{1,2} \ar[l] \ar@{-}[r] & {\tB_{a_2,e^{+}}}^{1,3} \ar@{-}[l] \ar@{-}[r] & {\tB_{a_1,e^{+}}}^{2,3} \ar@{-}[l]
}$$
$$\xymatrix{
{\tB_{b_1,e^+}}^{1,2,3} & {\tB_{b_2,e^+}}^{2,3} \ar[l] \ar@{-}[r] & {\tB_{b_3,e^+}}^{1,3} \ar@{-}[l] \ar@{-}[r] & {\tB_{b_4,e^+}}^{1,2} \ar@{-}[l]
}$$
Evidently, these are not isomorphic as $W$-graphs.

In this paper $W$-graphs are drawn with the following conventions: vertices are labeled by canonical basis elements and descent sets appear as superscripts; an edge with no arrow indicates that $\mu =1$ and neither descent set contains the other; an edge with an arrow indicates that $\mu =1$ and the descent set of the arrow head strictly contains that of the arrow tail; no edge indicates that $\mu = 0$ or the descent sets are the same.
\end{example}

For the remainder of this paper, let $J = J'_{n-1}$ ($r$ = 1) since this is
preferable for comparing $\H \tsr_J E$ with $(\pH \tsr_\H E)_1$ (see  \textsection\ref{ss affine tensor V}, below).
See  \textsection\ref{ss sb} for how to go back and forth between the $J'_{n-1}$ and $J_{n-1}$ pictures.

\subsection{}
\label{ss affine tensor V}
There is another $\u$-analogue of tensoring with $V$ that comes from the extended affine Hecke algebra $\eH$.  See \textsection \ref{ss affinecells} for a brief introduction to this algebra.

The module $\pH \tsr_\H E$ is a $\u$-analogue of $\ZZ \pW \tsr_{\ZZ \pW_{K_0}} E$, which, together with the following proposition, shows that $(\pH \tsr_\H E)_1$ is a $\u$-analogue of $V \tsr E$.
\begin{proposition}
\label{p affine u=1}
The correspondence
\be \begin{array}{ccc}
\Res_{\ZZ\S_n} \ZZ \pW \tsr_{\ZZ \pW_{K_0}} E & \cong &\ZZ[{x_1},\ldots, {x_n}] \tsr_\ZZ E,\\
y^\lambda \bt e & \longleftrightarrow & x^\lambda \tsr e,
\end{array}
\ee is a degree-preserving isomorphism of $\ZZ \S_n$-modules, natural in $E$, where $\S_n$ acts on the polynomial ring by permuting the variables.
\end{proposition}
\begin{proof}
Recalling that $\eW = Y \rtimes W$ with $W$ acting on $Y$ by permuting the coordinates, we have $s_i (y^\lambda \bt e) = s_i y^\lambda s_i \bt s_i e = y^{s_i(\lambda)} \bt s_i e$ and $s_i(x^\lambda \tsr e) = x^{s_i( \lambda)} \tsr s_i e$. 
\end{proof}

\begin{example}
 To compare with the $W$-graphs in Example \ref{ex 2wgraph}, here is the $W$-graph on $(\pH \tsr_\H e^{+})_1$.  In this case it is isomorphic to the $W$-graph on $\H\tsr_{{J'_{n-1}}}e^{+}$, but this is not true in general as can be seen by comparing Figures \ref{f VVe+} and \ref{f VVe+ affine}.
$$\xymatrix{
{\tB_{a_4\pi,e^{+}}}^{1,2,3} & {\tB_{a_3\pi,e^{+}}}^{1,2} \ar[l] \ar@{-}[r] & {\tB_{a_2\pi,e^{+}}}^{1,3} \ar@{-}[l] \ar@{-}[r] & {\tB_{a_1\pi,e^{+}}}^{2,3} \ar@{-}[l]}.$$
\end{example}

  The general relationship between $(\pH \tsr_\H E)_1$ and $\H\tsr_{{J'_{n-1}}}E $ can be explained as a special case of a $W$-graph version of Mackey's formula due to Howlett and Yin \cite[\textsection 5]{HY2}, which we now recall.

Let $\Gamma$ be a $W_I$-graph, and $K, I \subseteq S$.  Put $F = A \Gamma$.
Let $\leftexp{K}{W}^I$ be the set of minimal double coset representatives $\{d : d$ of minimal length in   $W_KdW_I\}$. For each $d\in \leftexp{K}{W}^I$, the \emph{d-subgraph} of (the $W_K$-graph on) $\Res_K\H \tsr_I F$ is $\{\tB_{wd,\gamma} :\ w\in W_K^L,  L = {K\cap dId^{-1}}, \gamma\in\Gamma\}$.

For any $d \in \leftexp{K}{W}^I$, let $L = K \cap dId^{-1}$.  Then $d^{-1} L d = d^{-1} Kd \cap I \subseteq I$ so the restriction $\Res_{d^{-1}Ld} F$ makes sense.  This $W_{d^{-1}Ld}$-graph naturally gives rise to a $W_L$-graph, denoted $d\Gamma$, obtained by conjugating descent sets by $d$. Explicitly, the descent set of a vertex $d\gamma$ of $d\Gamma$ is
\be L(d\gamma) = \{dsd^{-1}: s \in L(\gamma) \subseteq I \text{ and } dsd^{-1} \in K\} \subseteq L. \ee
The edge weights of $d\Gamma$ are the same as those of $\Gamma:\ \mu(d\delta,d\gamma) = \mu(\delta,\gamma)$ for all $\delta, \gamma \in \Gamma$.

\begin{theorem}[Howlett, Yin \cite{HY2}]\label{t mackey}
The $d$-subgraphs of $\Res_K\H \tsr_I F$  partition its canonical basis. Each $d$-subgraph is a union of cells and is isomorphic to $\H_K \tsr_L dF$ ($L = K \cap dId^{-1}$) as a $W_K$-graph via the correspondence $\tB_{wd,\gamma}\leftrightarrow\tB_{w,d\gamma}, w \in W^L_K$.
\end{theorem}

\begin{remark}
It is probably the case that each $d$-subgraph is a cellular subquotient rather than just a union of cells, however this is not proven in \cite{HY2}. This issue does not come up, however, because in the applications in this paper we can easily show that the $d$-subgraph is a cellular subquotient and sometimes the stronger statement that it is a cellular quotient or submodule.
\end{remark}
In the present application, put $K = I = \{s_1,\ldots,s_{n-1}\}$.  Then $\pi \in \leftexp{K}{W}^I$ and the $\pi$-subgraph of $\Res_\H \eH \tsr_\H E$ is $\{\tB_{a_k \pi,\gamma} : k\in [n],\gamma\in\Gamma\}$ since $K \cap \pi I \pi^{-1} = J'_{n-1}$. This is isomorphic as a $W$-graph to $\H \tsr_{J'_{n-1}} \pi E$.  The $W_{J'_{n-1}}$-graph $\pi E$ is just $\Res_{J_{n-1}} E$, with each element of its descent sets shifted up by one.  We have proved the following.
\begin{proposition}\label{p affine equals resind}
The $W$-graphs  $(\pH \tsr_\H E)_1$ and $\H \tsr_{J'_{n-1}} \pi E$ are isomorphic.
\end{proposition}

\begin{remark}
  Though this suggests that  the $W$-graph versions of $V \tsr E$,  $(\pH \tsr_\H E)_1$ and $\H \tsr_{J'_{n-1}} E$, behave in essentially the same way, some care must be taken. At $\u =1$, $\H \tsr_{J'_{n-1}} \pi E$ is not isomorphic to $V \tsr_\ZZ \pi E$ using Proposition \ref{p gk} since $\pi E$ is only a $\ZZ W_{J'_{n-1}}$-module, not a $\ZZ W$-module. Thus $(\pH \tsr_\H E)_1|_{\u=1}$ and $(\H \tsr_{J'_{n-1}} E)|_{\u=1}$ are only isomorphic to $V \tsr E$ by the rather different looking routes Proposition \ref{p affine u=1} and Corollary \ref{c indresiso}.
\end{remark}
\subsection{}

\label{ss reduced non-reduced}

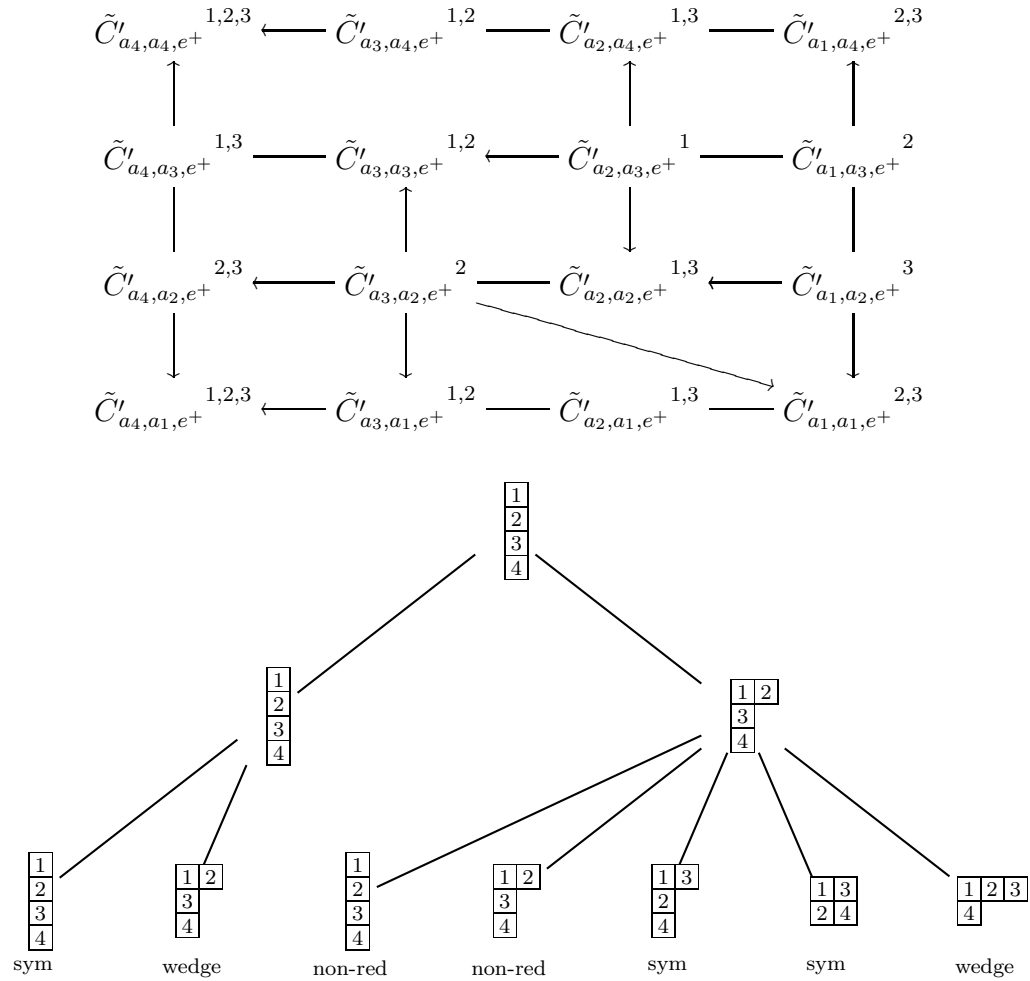
\begin{figure}
\begin{center}
$$
\xymatrix{
{\tB_{a_4,a_4,e^{+}}}^{1,2,3} & {\tB_{a_3,a_4,e^{+}}}^{1,2} \ar[l] \ar@{-}[r] & {\tB_{a_2,a_4,e^{+}}}^{1,3} \ar@{-}[l] \ar@{-}[r] & {\tB_{a_1,a_4,e^{+}}}^{2,3} \ar@{-}[l] \\
{\tB_{a_4,a_3,e^{+}}}^{1,3} \ar[u] \ar@{-}[r] & {\tB_{a_3,a_3,e^{+}}}^{1,2} \ar@{-}[l]  & {\tB_{a_2,a_3,e^{+}}}^{1} \ar[l] \ar@{-}[r] \ar[u] \ar[d] & {\tB_{a_1,a_3,e^{+}}}^{2} \ar@{-}[l] \ar[u] \ar@{-}[d]\\
{\tB_{a_4,a_2,e^{+}}}^{2,3} \ar[d] \ar@{-}[u] & {\tB_{a_3,a_2,e^{+}}}^{2} \ar[l] \ar[u] \ar@{-}[r] \ar[drr] \ar[d] & {\tB_{a_2,a_2,e^{+}}}^{1,3} & {\tB_{a_1,a_2,e^{+}}}^{3} \ar[l] \ar[l] \ar@{-}[u] \ar[d]\\
{\tB_{a_4,a_1,e^{+}}}^{1,2,3} & {\tB_{a_3,a_1,e^{+}}}^{1,2} \ar[l] \ar@{-}[r] & {\tB_{a_2,a_1,e^{+}}}^{1,3} \ar@{-}[l] \ar@{-}[r] & {\tB_{a_1,a_1,e^{+}}}^{2,3} \ar@{-}[l]\\
}$$
\end{center}
\begin{pspicture}(10,6.7){
\psset{unit=1cm}
\tiny
\newdimen\hcent
\hcent=5cm
\newdimen\ycor
\ycor=5.7cm
\advance\ycor by 0pt
\newdimen\hspacedim
\hspacedim=100pt
\newdimen\xcor
\xcor=\hcent
\newdimen\temp
\temp=140pt
\multiply \temp by 0 \divide \temp by 2
\advance\xcor by -\temp
\hoogte=9pt
\breedte=9pt
\dikte=0.2pt
\rput(\xcor,\ycor){\rnode{v1h1}
{
\begin{Young}
1\cr2\cr3\cr4\cr
\end{Young}
}
}
\advance\xcor by \hspacedim
\advance\ycor by -70pt
\hspacedim=180pt
\newdimen\xcor
\xcor=\hcent
\newdimen\temp
\temp=180pt
\multiply \temp by 1 \divide \temp by 2
\advance\xcor by -\temp

\rput(\xcor,\ycor){\rnode{v2h1}
{
\begin{Young}
1\cr2\cr3\cr4\cr
\end{Young}
}
}

\advance\xcor by \hspacedim
\rput(\xcor,\ycor){\rnode{v2h2}
{
\begin{Young}
1&2\cr3\cr4\cr
\end{Young}
}
}

\advance\xcor by \hspacedim
\advance\ycor by -70pt
\newdimen\hspacedim
\hspacedim=60pt
\newdimen\xcor
\xcor=\hcent
\newdimen\temp
\temp=60pt
\multiply \temp by 6 \divide \temp by 2
\advance\xcor by -\temp

\newdimen\ytmp
\ytmp=-0.1cm

\rput(\xcor,\ycor){\rnode{v3h1}
{
\begin{Young}
1\cr2\cr3\cr4\cr
\end{Young}
}
}
\rput(\xcor,\ytmp){sym}

\advance\xcor by \hspacedim
\rput(\xcor,\ycor){\rnode{v3h2}
{
\begin{Young}
1&2\cr3\cr4\cr
\end{Young}
}
}
\rput(\xcor,\ytmp){wedge}

\advance\xcor by \hspacedim
\rput(\xcor,\ycor){\rnode{v3h3}
{
\begin{Young}
1\cr2\cr3\cr4\cr
\end{Young}
}
}
\rput(\xcor,\ytmp){non-red}

\advance\xcor by \hspacedim

\rput(\xcor,\ycor){\rnode{v3h4}
{
\begin{Young}
1&2\cr3\cr4\cr
\end{Young}
}
}
\rput(\xcor,\ytmp){non-red}

\advance\xcor by \hspacedim
\rput(\xcor,\ycor){\rnode{v3h5}
{
\begin{Young}
1&3\cr2\cr4\cr
\end{Young}
}
}
\rput(\xcor,\ytmp){sym}

\advance\xcor by \hspacedim
\rput(\xcor,\ycor){\rnode{v3h6}
{
\begin{Young}
1&3\cr2&4\cr
\end{Young}
}
}
\rput(\xcor,\ytmp){sym}

\advance\xcor by \hspacedim
\rput(\xcor,\ycor){\rnode{v3h7}
{
\begin{Young}
1&2&3\cr4\cr
\end{Young}
}
}
\rput(\xcor,\ytmp){wedge}

\ncline{-}{v1h1}{v2h1}
\ncline{-}{v1h1}{v2h2}
\ncline{-}{v2h1}{v3h1}
\ncline{-}{v2h1}{v3h2}
\ncline{-}{v2h2}{v3h3}
\ncline{-}{v2h2}{v3h4}
\ncline{-}{v2h2}{v3h5}
\ncline{-}{v2h2}{v3h6}
\ncline{-}{v2h2}{v3h7}
}
\end{pspicture}
\caption{The $W$-graph on $\H \tsr_{J'_{n-1}} \H \tsr_{J'_{n-1}} e^+$ and the graph $G$ of \textsection \ref{ss locallabels}.  The vertices of the tree $G$ are marked by local labels. Each cell in the $W$-graph corresponds to the path from a leaf to the root that is its local sequence.}
\label{f VVe+}
\end{figure}

\begin{figure}
$$
\xymatrix{
{\tB_{a_4\pi,a_4\pi,e^{+}}}^{1,2,3} & {\tB_{a_3\pi,a_4\pi,e^{+}}}^{1,2} \ar[l] \ar@{-}[r] & {\tB_{a_2\pi,a_4\pi,e^{+}}}^{1,3} \ar@{-}[l] \ar@{-}[r] & {\tB_{a_1\pi,a_4\pi,e^{+}}}^{2,3}\\
{\tB_{a_4\pi,a_3\pi,e^{+}}}^{1,2,3} & {\tB_{a_3\pi,a_3\pi,e^{+}}}^{1,2} \ar[l]  & {\tB_{a_2\pi,a_3\pi,e^{+}}}^{1,3} \ar@{-}[l] \ar@{-}[r] & {\tB_{a_1\pi,a_3\pi,e^{+}}}^{2,3}\\
{\tB_{a_4\pi,a_2\pi,e^{+}}}^{1,3} \ar@{-}[d] \ar@{-}[r] \ar[u] & {\tB_{a_3\pi,a_2\pi,e^{+}}}^{1,2} & {\tB_{a_2\pi,a_2\pi,e^{+}}}^{1} \ar[l] \ar[d] \ar[u] & {\tB_{a_1\pi,a_2\pi,e^{+}}}^{2} \ar@{-}[l] \ar[u] \ar@{-}[d]\\
{\tB_{a_4\pi,a_1\pi,e^{+}}}^{2,3} & {\tB_{a_3\pi,a_1\pi,e^{+}}}^{2} \ar[l] \ar[u] \ar@{-}[r] & {\tB_{a_2\pi,a_1\pi,e^{+}}}^{1,3} & {\tB_{a_1\pi,a_1\pi,e^{+}}}^{3} \ar[l]\\
}$$

\begin{pspicture}(10,6.7){\tiny
\newdimen\hcent
\hcent=5cm
\newdimen\ycor
\ycor=5.8cm
\advance\ycor by 0pt
\newdimen\hspacedim
\hspacedim=160pt
\newdimen\xcor
\xcor=\hcent
\newdimen\temp
\temp=140pt
\multiply \temp by 0 \divide \temp by 2
\advance\xcor by -\temp
\hoogte=10pt
\breedte=10pt
\dikte=0.2pt
\rput(\xcor,\ycor){\rnode{v1h1}{
\begin{Young}
1\cr2\cr3\cr4\cr
\end{Young}
}
}
\advance\xcor by \hspacedim
\advance\ycor by -70pt
\newdimen\hspacedim
\hspacedim=180pt
\newdimen\xcor
\xcor=\hcent
\newdimen\temp
\temp=180pt
\multiply \temp by 1 \divide \temp by 2
\advance\xcor by -\temp
\rput(\xcor,\ycor){\rnode{v2h1}{
\begin{Young}
-3\cr2\cr3\cr4\cr
\end{Young}
}}
\advance\xcor by \hspacedim
\rput(\xcor,\ycor){\rnode{v2h2}{
\begin{Young}
-3&2\cr3\cr4\cr
\end{Young}
}}
\advance\xcor by \hspacedim
\advance\ycor by -70pt
\newdimen\hspacedim
\hspacedim=60pt
\newdimen\xcor
\xcor=\hcent
\newdimen\temp
\temp=60pt
\multiply \temp by 6 \divide \temp by 2
\advance\xcor by -\temp

\newdimen\ytmp
\ytmp=\ycor
\advance \ytmp by -27pt
\rput(\xcor,\ycor){\rnode{v3h1}{
\begin{Young}
\ng3\cr2\cr3\cr4\cr
\end{Young}
}}
\rput(\xcor,\ytmp){non-red}

\advance\xcor by \hspacedim
\rput(\xcor,\ycor){\rnode{v3h2}{
\begin{Young}
\ng3&2\cr3\cr4\cr
\end{Young}
}}
\rput(\xcor,\ytmp){non-red}

\advance\xcor by \hspacedim
\rput(\xcor,\ycor){\rnode{v3h3}{
\begin{Young}
-2\cr1\cr3\cr4\cr
\end{Young}
}}
\rput(\xcor,\ytmp){sym}

\advance\xcor by \hspacedim
\rput(\xcor,\ycor){\rnode{v3h4}{\begin{Young}
-2&1\cr3\cr4\cr
\end{Young}}}
\rput(\xcor,\ytmp){wedge}

\advance\xcor by \hspacedim
\rput(\xcor,\ycor){\rnode{v3h5}{\begin{Young}
-2&3\cr1\cr4\cr
\end{Young}}}
\rput(\xcor,\ytmp){sym}

\advance\xcor by \hspacedim
\rput(\xcor,\ycor){\rnode{v3h6}{
\begin{Young}
-2&3\cr1&4\cr
\end{Young}
}}
\rput(\xcor,\ytmp){sym}

\advance\xcor by \hspacedim
\rput(\xcor,\ycor){\rnode{v3h7}{
\begin{Young}
-2&1&3\cr4\cr
\end{Young}
}}
\rput(\xcor,\ytmp){wedge}

\ncline{-}{v1h1}{v2h1}
\ncline{-}{v1h1}{v2h2}
\ncline{-}{v2h1}{v3h1}
\ncline{-}{v2h1}{v3h2}
\ncline{-}{v2h2}{v3h3}
\ncline{-}{v2h2}{v3h4}
\ncline{-}{v2h2}{v3h5}
\ncline{-}{v2h2}{v3h6}
\ncline{-}{v2h2}{v3h7}
}
\end{pspicture}
\caption{The $W$-graph on $(\pH\tsr_\H(\pH\tsr_\H e^+)_1)_1$ and the graph $G$ of \textsection \ref{ss locallabels}, with the labeling conventions of \textsection\ref{ss affinecells}.}
\label{f VVe+ affine}
\end{figure}

Let $\Gamma$ be a $W$-graph, and put $E = A \Gamma$, $F = \Res_J A \Gamma$, $\tE^2:= \H\tsr_{J}\H\tsr_{J} A\Gamma$.

We will show that $\tE^2$  decomposes into what we call a reduced and non-reduced part.  Towards this end, consider the exact sequence

\be \label{e mackey}\xymatrix@R=.2cm{0 \ar[r]& F \ar[r]_<<<<{{\unit}} & \Res_J \H \tsr_J F \ar[r]_<<<<{\tau} & \H_{J}\tsr_{J'_{n-2}}\Res_{J'_{n-2}} F  \ar[r]& 0. \\
& \gamma \ar@{|->}[r] & 1 \bt \gamma \ar@{|->}[r]& 0 &\\
&& T_{a_k}\bt\gamma \ar@{|->}[r]& T_{s_{k-1} \ldots s_{2}} \bt \gamma &}\ee

By Proposition \ref{p res ind} the image of $\unit$ is a cellular submodule.  The map $\tau$ induces an isomorphism of $W_J$-graphs $\Res_J \H \tsr_J F / \im(\unit) \cong  \H_{J}\tsr_{J'_{n-2}}\Res_{J'_{n-2}} F $; given that the  sequence is exact, this is equivalent to taking canonical basis elements to canonical basis elements or to 0. That $\tau$ is an isomorphism can be seen directly by observing that it takes standard basis elements of $\H\tsr_JF$ to standard basis elements of $\H_J \tsr_{J'_{n-2}} F$ or to 0, takes the lattice $A^- \H \tsr_J \Gamma$ to the lattice $A^- \H_J \tsr_{J'_{n-2}} \Gamma$, and intertwines the involutions $\br{\cdot}$.

This decomposition also comes another application of the $W$-graph version of Mackey's formula (Theorem \ref{t mackey}).  For this application, put $K = I = J (= \{s_2,\ldots,s_{n-1}\})$.  Then $\leftexp{K}{W}^I = \{1,s_{1}\}$.  The $1$-subgraph of $\Res_J \H \tsr_J F$ is $\{\tB_{1,\gamma} : \gamma\in\Gamma\}$ and the $s_{1}$-subgraph is $\{\tB_{w s_{1},\gamma} :\ w\in W_J^{J'_{n-2}}, \gamma\in\Gamma\}$.  These are isomorphic as $W_J$-graphs to $\H_J \tsr_J F = F$ and $\H_{J}\tsr_{J'_{n-2}}\Res_{J'_{n-2}} F$ respectively (since we have $d^{-1} Ld = L$ for all $d$, $d F$ and $F$ are identical).

Next, tensor (\ref{e mackey}) with $\H$ to obtain
 \be \label{e mackey2}\xymatrix{0 \ar[r]& \H \tsr_J A \Gamma \ar[r]_<<<<{\H \tsr_J{\unit}} & \H \tsr_J \H \tsr_J A\Gamma \ar[r]_<<<<{\H \tsr_J \tau} & \H\tsr_{J'_{n-2}} A \Gamma  \ar[r]& 0.}\ee
Put $\tF^2 = \H\tsr_{J'_{n-2}} A \Gamma$.    The quotient $\tF^2$ (resp. the submodule $\H \tsr_J A \Gamma$) is the \emph{reduced} (resp. \emph{non-reduced}) part of $\tE^2$.
\begin{proposition}
\label{p red notred}
The submodule and quotient of $\tE^2$ given by (\ref{e mackey2}) are cellular and the maps in (\ref{e mackey2}) take canonical basis elements to canonical basis elements or to 0.
\end{proposition}
\begin{proof}
This follows from the application of Theorem \ref{t mackey} described above and Proposition \ref{CBSubquotientprop}.
\end{proof}

\begin{example}
The non-reduced part of $\tE^2$ for $E = e^+$ is the bottom row of the $W$-graph in Figure \ref{f VVe+}. The cells comprising it are labeled ``non-red'' below the tree.
\end{example}

\subsection{}
\label{ss red notred u=1}
Let us determine what the decomposition of $\tE^2$ into reduced and non-reduced parts becomes at $\u = 1$.

\begin{proposition}
\label{p mackey2 u=1}
At $\u=1$, (\ref{e mackey2}) becomes
\be \label{e mackey2 u=1}
\xymatrix@R=.2cm{0 \ar[r]& V \tsr E \ar[r] & V \tsr V \tsr E \ar[r] & T^2_\text{red} V \tsr E \ar[r]& 0. \\
& x_k \tsr \gamma \ar@{|->}[r] & x_k \tsr x_k \tsr \gamma \ar@{|->}[r]& 0 &\\
&& x_i \tsr x_j \tsr \gamma \ar@{|->}[r]& x_i \tsr x_j \tsr \gamma, &}
\ee
where $i \neq j$ and $T^2_\text{red} V := \ZZ \{ x_i \tsr x_j : i \neq j, i,j \in [n]\} \subseteq V \tsr V$.
\end{proposition}

To see this, first define $a_{k,l} = s_{k-1} \ldots s_{1}s_{l-1}\ldots s_{2}$ for $k\in [n],\ l \in \{2,\ldots,n\}$; then
\be \label{e aklcoset}
\begin{array}{rcl}
a_{k,l} \cdot s_1 &=& a_{l,k+1} \text{ if } k < l, \\
W^{J'_{n-2}} &=& \{a_{k,l} : k \in [n],\ l \in \{2, \ldots,n\}\}, \\
W^{S\backslash {s_{2}}} s_{1} &=& \{a_{k,l} : k \geq l>1\}, \text{ and}\\
 W^{S \backslash {s_{2}}} &=& \{a_{k,l} : k < l\}.
\end{array}
\ee

Apply Corollary \ref{c indresiso} twice to obtain
\be
\label{e vve at u=1}
\begin{array}{rcccl}
\tE^2|_{\u=1} & \cong & \ZZ\S_n\tsr_{\ZZ\S_{n-1}}V \tsr E & \cong & V \tsr V\tsr E\\
a_k \bt a_l \bt \gamma & \leftrightarrow & a_k \bt (x_l\tsr a_l(\gamma)) & \leftrightarrow & \left\{
\begin{array}{ll}
x_k \tsr x_k \tsr a_k a_l(\gamma) & \text{ if } l = 1,\\
x_k \tsr x_l \tsr a_k a_l(\gamma) & \text{ if } k < l, \\
x_k \tsr x_{l-1} \tsr a_k a_l(\gamma) & \text { if } k \geq l > 1. \end{array}\right.
\end{array} \ee

\begin{proof}[Proof of Proposition \ref{p mackey2 u=1}]
The interesting part of the calculation is the following diagram
\be
\label{e red notred diagram}
\xymatrix{
\tE^2|_{\u=1} \ar@{<->}[d]_<<<<{\cong}\ar@{->>}[r]& \tF^2|_{\u=1} \ar@{<->}[d]_<<<<{\cong}  & a_k \bt a_l \bt \gamma \ar@{|->}[r]\ar@{<->}[d] &  \ar@{<->}[d] a_{k,l} \bt \gamma \\
V \tsr V \tsr E \ar@{->>}[r]& T^2_\text{red} V \tsr E & x_k \tsr a_k (x_l) \tsr a_k a_l(\gamma) \ar@{|->}[r] & x_k \tsr a_k(x_l) \tsr a_{k,l} s_1 (\gamma)
}
\ee
where $k \in [n], l \in \{2,\ldots,n\}$.

There is a slightly tricky point here: the left-hand isomorphism of (\ref{e red notred diagram}) comes from (\ref{e vve at u=1}), but the right-hand isomorphism does not come from a similar application of Proposition \ref{p gk}. However, Proposition \ref{p gk} also holds with the isomorphism $g \bt e \mapsto (g \bt 1) \tsr g ce$ replacing (\ref{e gk}), where $c \in G$ commutes with all of $K$. In this case we must choose $c = s_1$ (which commutes with $K = J'_{n-2}$) to make the diagram (\ref{e red notred diagram}) commute.
\end{proof}

\subsection{}
There is a similar decomposition of $\hE^2 := (\pH\tsr_\H(\pH\tsr_\H A\Gamma)_1)_1$ into a reduced and non-reduced part. Two applications of Proposition \ref{p affine equals resind} yield $\hE^2 = \H \tsr_J \pi(\H \tsr_J \pi E)$.

First, let us apply Theorem \ref{t mackey} to $\Res_{J_{n-1}} \H \tsr_{J'_{n-1}} \pi E $ analogously to the application in the previous subsection.    In this case $I = J'_{n-1}, \ K = J_{n-1}$, and therefore $\leftexp{K}{W}^I = \{1, a_n\}$.

 The $1$-subgraph is $\{\tB_{a_k, \pi \gamma} : k < n, \pi \gamma \in \pi \Gamma\}$ and spans a cellular submodule of $\Res_{J_{n-1}} \H \tsr_{J'_{n-1}} \pi E $.  This can be seen, for instance, by applying Proposition \ref{p ic lower ideal} with the order $\ord$ of \textsection \ref{ss IC basis} to obtain
 \be A \{\tB_{a_k ,\pi \gamma}: k < n,  \pi \gamma \in \pi \Gamma\} = A \{\T_{a_k ,\pi \gamma}: k < n, \pi \gamma \in \pi \Gamma\}; \ee
 it is clear that this $A$-span of $\T$'s is left stable under the action of $\H_{J_{n-1}}$. Now this submodule is isomorphic to $\H_{J_{n-1}} \tsr_{J_{n-1}\backslash s_1} \pi E $ (as a $W_{J_{n-1}}$-graph) by Theorem \ref{t mackey}.

 The $a_n$-subgraph is $\{\tB_{a_n, \pi \gamma} : \pi \gamma \in \pi \Gamma\}$ and spans a cellular quotient since the only other $d$-subgraph spans a submodule.  This quotient is isomorphic to $a_n \pi E$ as a $W_{J_{n-1}}$-graph.  Moreover, $a_n \pi E$ is exactly $\Res_{J_{n-1}} E$ as $L = K \cap a_n I a_n^{-1} = K =
 {J_{n-1}}$.  The following exact sequence summarizes what we have so far.
 \be \label{e mackey affine}\xymatrix@R=.2cm{
0 \ar[r]& \H_{J_{n-1}} \tsr_{J_{n-1}\backslash s_1} \pi E \ar[r] & \Res_{J_{n-1}} \H \tsr_{J'_{n-1}} \pi E \ar[r] & \Res_{J_{n-1}} E \ar[r]& 0.
}\ee

Applying $\pi$ to the $W_{J_{n-1}}$-graphs  in this sequence to obtain $W_{J'_{n-1}}$-graphs (as explained before Theorem \ref{t mackey}) and then tensoring with $\H$ yields
 \be \label{e mackey2 affine}\xymatrix@C=.24cm@R=.47cm{0 \ar[r]& \H \tsr_{J'_{n-1}} \pi (\H_{J_{n-1}} \tsr_{J_{n-1}\backslash s_1} \pi E) \ar[r] \ar@{<->}[d]^{\cong} & \H \tsr_{J'_{n-1}} \pi (\H \tsr_{J'_{n-1}} \pi E) \ar[r]\ar@{<->}[d]^{\cong} & \H \tsr_{J'_{n-1}} \pi E \ar[r]\ar@{<->}[d]^{\cong}& 0\\
 0 \ar[r] & \H \tsr_{J'_{n-2}} \Res_{J'_{n-2}}\pi^2 E \ar[r] & (\pH\tsr_\H(\pH\tsr_\H E)_1)_1 \ar[r] & (\pH \tsr_\H E)_1 \ar[r]& 0, }\ee
where $\Res_{J'_{n-2}}\pi^2 E$ is the $W_{J'_{n-2}}$-graph obtained from $\Res_{J_{n-2}} E$ by increasing descent set indices by 2.  The leftmost isomorphism  comes from the isomorphism of Coxeter group pairs $(W_{J_{n-1}\backslash s_1},W_{J_{n-1}}) \cong (W_{J'_{n-2}},W_{J'_{n-1}})$ given by conjugation by $\pi$. The other two isomorphisms are applications of Proposition \ref{p affine equals resind} .

The submodule $\hF^2 :=\H \tsr_{J'_{n-2}} \Res_{J'_{n-2}}\pi^2 E$ (resp. the quotient $(\pH \tsr_\H E)_1$) is the \emph{reduced} (resp. \emph{non-reduced}) part of $\hE^2$.   We have proved the following analogue of Proposition \ref{p red notred}.
\begin{proposition}\label{p red notred affine}
The submodule and quotient of $\hE^2$ given by (\ref{e mackey2 affine}) are cellular
and the maps in (\ref{e mackey2 affine}) take canonical basis elements to canonical basis elements or to 0.
\end{proposition}

\begin{example}
The non-reduced part of $\hE^2$ for $E = e^+$ is the top row of the $W$-graph in Figure \ref{f VVe+ affine}. The cells comprising it are labeled ``non-red'' below the tree.
\end{example}

At $\u=1$, the decomposition (\ref{e mackey2 affine}) becomes
\be \hE^2|_{\u=1} \cong V \tsr V \tsr E \cong T^2_\text{red} V \tsr E \oplus V \tsr E\ee
(with the left-hand isomorphism from Proposition \ref{p affine u=1}),  but the computation is different from that of \textsection\ref{ss red notred u=1}. We omit the details.


\section{Decomposing $V \tsr V \tsr E$ and the functor $\tsymred$}
\label{s decomposing VVE}
In this section we study a $W$-graph version of the decomposition $V \tsr V \tsr E \cong S^2V\tsr E \oplus \Lambda^2V\tsr E$. Along the way, we come across a mysterious object, the sym-wedge functor $\tsymred$. At $\u = 1$, this is some kind of mixture of the functors $S^2_\text{red} V \tsr-$ and $\Lambda^2 V \tsr-$, where $S^2_\text{red} V = \ZZ \{ x_i \tsr x_j + x_j \tsr x_i : i \neq j \} \subseteq S^2 V \subseteq V \tsr V$.

\subsection{}

Let $\Lambda$ be the $W_{S \backslash s_2}$-graph on $\H_{S \backslash s_2} \tsr_{J'_{n-2}} A \Gamma$ obtained from Theorem \ref{t HY canbas exists}.
For any $W$-graph $\ce$ and $s\in S$, define $\ce^-_s = \{\gamma \in \ce: s \in L(\gamma)\}$ and $\ce^+_s = \{\gamma \in \ce: s \notin L(\gamma)\}$.  In this case, $\Lambda^-_{s_1} = \{\tB_{s_1,\gamma}:\gamma \in \Gamma\}$, and $\Lambda^+_{s_1} =  \{\tB_{1,\gamma}:\gamma \in \Gamma\}$ as $L(w, \gamma) = L(w) \cup L(\gamma)$.  Also note that $\tB_{1,\gamma} = \C_{1} \bt \gamma$ and $\tB_{s_1,\gamma} = \C_{s_1} \bt \gamma$.

It is clear that in the case $\Gamma = \Gamma_{W_{J'_{n-2}}}$, $A \Lambda^-_{s_1}$ is a cellular submodule of $A \Lambda$.  This is actually true in full generality as we will see shortly (Lemma \ref{l s1 descent submodule}).  Now define the \emph{sym-wedge} functor $\tsymred$ by $\tsymred A\Gamma = \H\tsr_{{S \backslash s_2}} A\Lambda^-_{s_1}$, with a $W$-graph structure coming from Theorem \ref{t HY canbas exists}.

\begin{theorem}\label{t twisted S2}
The $\H$-module $\tsymred A\Gamma$ is a cellular submodule of $\tF^2 := \H\tsr_{J'_{n-2}} A \Gamma$.
\end{theorem}
\begin{proof}
By Lemma \ref{l s1 descent submodule} (below), $A \Lambda^-_{s_1}$ is a cellular submodule of $A \Lambda$. Proposition \ref{CBSubquotientprop} shows that $\H \tsr_{S \backslash s_2} A\Lambda^-_{s_1}$ is a cellular submodule of $\H \tsr_{S \backslash s_2} A\Lambda$, and $\H \tsr_{S \backslash s_2} \Lambda$ and $\H\tsr_{J'_{n-2}} \Gamma$  give the same
$W$-graph structure on $\tF^2$ by Proposition \ref{p nested parabolic}.
\end{proof}

The sym-wedge functor was discovered by looking at examples.   The preceding proof sort of explains why such a cellular submodule should exist, but it is still somewhat surprising it does not agree with $S^2_\text{red} V\tsr-$ at $\u = 1$.  We will determine what $\tsymred A\Gamma$ is at $\u=1$ in \textsection \ref{ss u=1 tS2} and address its relation with $S^2_\text{red} V \tsr-$ and $\Lambda^2 V \tsr-$ in \textsection \ref{ss S2 vs twisted S2}.  It will be useful for us later    to know the following additional structure possessed by $\tsymred$.

\begin{proposition} \label{p cbsubquotientprop tsymred}
The rule $E \mapsto \tsymred E$ is a functor $\tsymred : \H$-$\Mod \to \H$-$\Mod$. Moreover,  if $E = A\Gamma$ for some $W$-graph $\Gamma$, then taking cellular submodules or quotients of $E$ gives rise to cellular submodules and quotients of $\tsymred E$ in the same way induction does in Proposition \ref{CBSubquotientprop}.
\end{proposition}
\begin{proof}
As explained above, the proposed functor $\tsymred$ is the composition
\be \H\text{-}\Mod \xrightarrow{\Res_{J'_{n-2}}} \H_{J'_{n-2}}\text{-}\Mod \xrightarrow{\H_{S \backslash s_2}\tsr-} \H_{S \backslash s_2}\text{-}\Mod \xrightarrow{\zeta} \H_{S \backslash s_2}\text{-}\Mod \xrightarrow{\H\tsr-} \H\text{-}\Mod, \ee
where $\zeta(F)$ is the kernel of $F \xrightarrow{m_{\C_{s_1}-[2]}} F$ and $m_{h}$ is left multiplication by $h$ (by Lemma \ref{l s1 descent submodule}, $m_{\C_{s_1}-[2]}$ is an $\H_{S \backslash s_2}$-module homomorphism and its kernel equals $A\Lambda^-_{s_1}$ in the case $F = A\Lambda$). Thus it suffices to show that $\zeta$ is a functor and respects cellular subquotients as claimed.

Let $F$ and $F^*$ be $W_{S \backslash s_2}$-graphs and $f: F \to F^*$ be an $\H_{S \backslash s_2}$-module homomorphism. As $f\  m_{\C_{s_1}-[2]} = m_{\C_{s_1}-[2]}  f$, $f(\ker(m_{\C_{s_1}-[2]})) \subseteq \ker(m_{\C_{s_1}-[2]})$. Thus $f \mapsto f|_{\ker(m_{\C_{s_1}-[2]})}$ defines $\zeta$ on morphisms and this   certainly respects composition of morphisms.
%

For the second statement, just observe that if $\Lambda$ is a $W_{S \backslash s_2}$-graph  and $\ce \subseteq \Lambda$ spans a cellular submodule, then $\zeta(A\ce)$ is the intersection of the cellular submodules $A\ce$ and $A\Lambda^-_{s_1}$, which is a cellular submodule of $\zeta(A\Lambda) = A\Lambda^-_{s_1}$. Similarly, if $\ce^*$ is the vertex set $\Lambda \backslash \ce$, then $\zeta(A\ce^*) = A(\ce^* \cap \Lambda^-_{s_1})$, which is the cellular quotient $A\Lambda^-_{s_1} / \zeta(A\ce)$ of $A\Lambda^-_{s_1}$.
\end{proof}

\begin{lemma} \label{l s1 descent submodule}
For any $W$-graph $\Lambda$ and $s \in S$, the kernel of the map of abelian groups $m_{\C_s-[2]} :A\Lambda \to A\Lambda$ (where $m_h$ is left multiplication by $h$) is equal to $A\Lambda^-_s$. If $s$ commutes with $t$ for all $t \in S$ and $F$ is any $\H$-module, then $m_{\C_s-[2]}:F \to F$ is an $\H$-module  homomorphism.  Therefore, if $\Lambda$ is a $W_{S\backslash s_2}$-graph, then $A \Lambda^-_{s_1}$ is a cellular submodule of $A \Lambda$.
\end{lemma}
\begin{proof}
Certainly any $h \in A\Lambda^-_s$ is in the kernel of $m_{\C_s - [2]}$. To see that the kernel is no bigger, let $h = \sum_{\lambda \in \Lambda} c_\lambda \lambda$ ($c_\lambda \in A$) be an element of $A\Lambda$ satisfying $(\C_s-[2]) h = 0$. We may assume   that $c_\lambda = 0$ for $\lambda \in \Lambda^-_s$. Also, by multiplying the $c$'s by some power of $\u$, we may assume that $c_\lambda \in A^-$ for all $\lambda$ and $c_\lambda \notin \ui A^-\Lambda$ for at least one $\lambda$. Then computing mod $A^-\Lambda$, we have
\be 0 = \sum\limits_{\lambda \in \Lambda} c_\lambda \left(\sum\limits_{\{\delta : s \in L(\delta)\}} \mu(\delta, \lambda)\delta\right) - [2] \sum\limits_{\lambda \in \Lambda} c_\lambda \lambda \equiv -\u \sum\limits_{\lambda \in \Lambda} c_\lambda \lambda. \ee
Therefore $c_\lambda \in \ui A^- \Lambda$ for all $\lambda$, contradicting the earlier assumption.

The second statement is a special case of the fact that $m_h$ is an $\H$-module homomorphism whenever $h$ is in the center of $\H$.
\end{proof}

\subsection{}
\label{ss u=1 tS2}
To better understand the functor $\tsymred$, let us determine what it becomes at $\u = 1$.

\begin{proposition}
The image of $\tsymred E |_{\u=1}$ under the isomorphism $\tF^2 |_{\u=1} \cong T^2_\text{red} V\tsr E$ of (\ref{e red notred diagram}) is $S^2_\text{red} V \tsr E$ (resp. $\Lambda^2 V \tsr E$) if $\Res_{W_{\{s_1\}}} E$ is a sum of copies of the trivial (resp. sign) representation.
\end{proposition}

\begin{proof}
Under the isomorphism $\tF^2 |_{\u=1} \cong T^2_\text{red} V\tsr E$, the standard basis for $\tsymred E$ coming from realizing it as $\H \tsr_{S \backslash s_2} A \Lambda^-_{s_1}$ (see the discussion before Theorem \ref{t twisted S2}) satisfies
\be \label{e Z2 u=1}
(T_{a_{k,l}} \bt_{S \backslash s_2} \C_{s_1} \bt_{J'_{n-2}} \gamma)|_{\u=1} = (a_{k,l} + a_{l,k+1} )\bt \gamma \longleftrightarrow x_k \tsr x_l \tsr a_{k,l} s_1 \gamma + x_l \tsr x_k\tsr a_{k,l}  \gamma  \quad (k < l),
\ee
where (\ref{e aklcoset}) has been used freely.
Therefore if $s_{1}$ acts trivially on $E$, then the rightmost expression in (\ref{e Z2 u=1}) becomes $(x_k \tsr x_l + x_l \tsr x_k) \tsr a_{k,l} \gamma$.  If $s_{1}$ acts by -1 on $E$, then it becomes $(-x_k \tsr x_l + x_l \tsr x_k) \tsr a_{k,l} \gamma$.  The proposition then follows, as $\ZZ \{a_{k,l} \gamma: \gamma \in \Gamma\} = E|_{\u=1}$.
\end{proof}

\subsection{}
A correct $W$-graph version of tensoring $S_{\text{red}}^2 V$ with $E$ is $\H \tsr_{S\backslash s_2} E$, and the projection $T^2_\text{red} V \tsr E \twoheadrightarrow S^2_\text{red} V \tsr E$ corresponds to
\be \label{e reduced v tsr v to reduced s^2} \tF^2 = \H\tsr_{J'_{n-2}} E = \H\tsr_{S\backslash s_2} \H_{S\backslash s_2}\tsr_{J'_{n-2}} E \xrightarrow{\tilde{\counit}(E)} \H\tsr_{S\backslash s_2} E, \ee where $\tilde{\counit} (E) = \H\tsr_{S\backslash s_2}\counit (E)$. This is justified by the following calculation at $\u = 1$.

\begin{proposition}
The module $\H \tsr_{S \backslash s_2} E$ is a $\u$-analogue of $S^2_\text{red} V \tsr E$ (via the right vertical map of the following diagram, to be defined) in a way so that the diagram commutes.
\be
\xymatrix@C=1cm{
(\H \tsr_{J'_{n-2}} E) |_{\u = 1} \ar@{<->}[d]^\cong\ar@{->>}[r]_<<<<<{\tilde{\counit}(E)}& (\H \tsr_{S \backslash s_2} E)|_{\u = 1} \ar@{<->}[d]^\cong\\
T^2_\text{red} V \tsr E \ar@{->>}[r] & S^2_\text{red} V \tsr E
}
\ee
\end{proposition}
\begin{proof}
Here we will think of $S^2_\text{red} V$ as the subspace $\ZZ \{x_k x_l : k \neq l\}$ of $(\ZZ[x_1,\ldots,x_n])_2$, and the map $T^2_\text{red} V \tsr E \to S^2_\text{red} V \tsr E$ as the one sending $x_k \tsr x_l$ to $x_k x_l$.   The right vertical map comes from an application of the modified Proposition \ref{p gk} (in which $g \bt e \mapsto (g \bt 1) \tsr g ce$ replaces (\ref{e gk}), where $c \in G$ commutes with all of $K$).  In this application, use $G = W, K = W_{S \backslash s_2}$, $c = s_1$. We have $\ZZ G \tsr_{\ZZ K} \ZZ \cong S^2_\text{red} V$ by $a_{k, l} \bt 1 \mapsto x_k x_l$ for $k < l$. It is straightforward to check that this is a $\ZZ G$-module homomorphism; the most interesting case is $s_k a_{k, k+1} \bt 1 = a_{k+1, k+1} \bt 1 = a_{k, k+1} s_1 \bt 1 = a_{k, k+1} \bt 1$, which matches $s_k (x_k x_{k+1}) = x_k x_{k+1}$.  It can be checked directly on the basis $\{a_{k,l} \bt \gamma: k \in [n], l \in \{2,\ldots,n\}, \gamma \in \Gamma\}$ of $(\H \tsr_{J'_{n-2}} E) |_{\u = 1}$  that the diagram commutes.
\end{proof}

\subsection{}
\label{ss diagram commutes}
It is immediate from Proposition \ref{p affine u=1} that the right-hand vertical map in the diagram below is a $\u$-analogue of the surjection $V \tsr V \tsr E \to S^2 V \tsr E$. Let us check that this is compatible with the projection $\tilde{\counit}(\pi^2 E)$ -- the $\u$-analogue of the projection $T^2_\text{red} V \tsr E \to S^2_\text{red} V \tsr E$. This amounts to checking that the following diagram commutes, where the top horizontal map is from (\ref{e mackey2 affine}) and the bottom horizontal map we take to be the inclusion of the $\pi^2$-subgraph of $\Res_\H \eH \tsr_\H E$.
\be
\xymatrix{
\H \tsr_{J'_{n-2}} \pi^2 E \ar[r]\ar@{->>}[d]_<<<<{\tilde{\counit}(\pi^2 E)} & (\pH \tsr_{\H} (\pH \tsr_{\H} E)_1)_1 \ar@{->>}[d]^<<<<{\counit((\pH \tsr_{\H} E)_1)} \\
\H \tsr_{S \backslash s_2} \pi^2 E \ar[r]& (\pH \tsr_{\H} E)_2
}
\ee
It is straightforward to check, given Theorem \ref{t mackey} and the derivation of (\ref{e mackey2 affine}), that standard basis elements behave as shown under the horizontal maps. This proves that the diagram commutes.
\be
\xymatrix{
\T_{a_{k,l}, \pi^2 \gamma} \ar@{|->}[r]\ar@{|->}[d] & \T_{a_k \pi, a_{l-1} \pi, \gamma}\ar@{|->}[d] & \T_{a_{k,l}, \pi^2 \gamma} \ar@{|->}[r]\ar@{|->}[d]& \ar@{|->}[d] \T_{a_k \pi, a_{l-1} \pi, \gamma} \\
\T_{a_{k,l}, \pi^2\gamma} \ar@{|->}[r] & \T_{a_{k,l} \pi^2, \gamma} & T_{a_{l-1,k}} \bt T_{s_1} (\pi^2\gamma) \ar@{|->}[r] & T_{a_{l-1,k}} \pi^2 \bt T_{s_{n-1}} \gamma}
\ee
The left-hand diagram is for $k < l$ and the right for $k \geq l > 1$.

This calculation will be used to show that the work we do in the next subsection for the $\H \tsr_J -$ version of tensoring with $V$ is also useful for the $(\pH \tsr_\H -)_1$ version.

\subsection{}
\label{ss S2 vs twisted S2}
In this subsection we will partially determine the projection $\tilde{\counit}(E)$  on canonical basis elements. Despite the fact that $\H \tsr_{S \backslash s_2} E$ is a $\u$-analogue of $S^2_\text{red}V \tsr E$ and $\tsymred E$ is not, our study of $\tsymred$ was not wasted.  It will be helpful for determining what $\tilde{\counit}(E)$ does to canonical basis elements.  This is not so easy to see directly, as it does not simply send canonical basis elements to canonical basis elements.

By Lemma \ref{l s1 descent submodule}, $A\Gamma^-_{s_1}$ is a cellular submodule of $\Res_{S \backslash s_2} A\Gamma$ with corresponding quotient $A\Gamma^+_{s_1}$, hence the exact sequence
\be 0 \to A\Gamma^-_{s_1} \to \Res_{S\backslash s_2} A\Gamma \to A\Gamma^+_{s_1} \to 0. \ee
Since $\tF^2, \tsymred A\Gamma, \symred A\Gamma$ only depend on $\Res_{S\backslash s_2} A\Gamma,$ this sequence yields the three columns in the diagram below.  The left column is exact by Proposition \ref{p cbsubquotientprop tsymred} and the other two are exact by exactness of induction. The left two squares commute because $\zeta$ (of the proof of Proposition \ref{p cbsubquotientprop tsymred}) of a morphism just restricts its domain, and the right two squares commute because $\counit$ is a natural transformation.
\be\label{e 5x3 diagram}
\xymatrix@C=1.5cm{0 \ar[d] & 0 \ar[d] & 0 \ar[d]\\
\tsymred A\Gamma^-_{s_1} \ar[r]\ar[d] & \H\tsr_{J'_{n-2}} A\Gamma^-_{s_1} \ar[d]\ar[r]_<<<<<<<{\tilde{\counit}(A\Gamma^-_{s_1})} & \H\tsr_{S\backslash s_2} A\Gamma^-_{s_1} \ar[d]\\
\tsymred A\Gamma \ar[r]\ar[d] & \H\tsr_{J'_{n-2}} A\Gamma \ar[d]\ar[r]_<<<<<<<{\tilde{\counit}(A\Gamma)}\ar[d] & \H\tsr_{S\backslash s_2} A\Gamma \ar[d]\\
\tsymred A\Gamma^+_{s_1} \ar[r]\ar[d] & \H\tsr_{J'_{n-2}} A\Gamma^+_{s_1} \ar[d]\ar[r]_<<<<<<<{\tilde{\counit}(A\Gamma^+_{s_1})} \ar[d]& \H\tsr_{S\backslash s_2} A\Gamma^+_{s_1} \ar[d]\\
0 & 0 & 0
}\ee

\newcommand{\Ltwo}{\ensuremath{\L'}}

\begin{lemma}
\label{l L to L*}
Given $w \in W^{J'_{n-2}}$, $\gamma \in \Gamma$,  suppose that either $s_1 \notin R(w)$ or $s_1 \notin L(\gamma)$.  Then $\tilde{\counit}(A\Gamma)(\tB_{w,\gamma})$, $\tB_{w, \gamma} \in \H \tsr_{J'_{n-2}} A \Gamma$
lies in the lattice $\Ltwo := A^- \H \tsr_{S \backslash s_2} \Gamma$.
\end{lemma}
\begin{proof}
First note that the standard basis for $\H \tsr_{J'_{n-2}} A \Gamma$ coming from realizing   $\H \tsr_{J'_{n-2}} A \Gamma$ as $\H \tsr_{S \backslash s_2} \H_{S \backslash s_2} \tsr_{J'_{n-2}} A \Gamma$ satisfies
\be
\label{e TvC mapsto}
\begin{array}{cccclr}
\T_{v, \tB_{1,\gamma}}& = & T_v \bt_{S \backslash s_2} 1 \bt_{J'_{n-2}} \gamma & \stackrel{\tilde{\counit}(A\Gamma)}{\longmapsto} &
T_v \bt_{S \backslash s_2} \gamma, \text{ and} \\
\T_{v, \tB_{s_1,\gamma}} &=&T_v \bt_{S \backslash s_2} \C_{s_1} \bt_{J'_{n-2}} \gamma &
\stackrel{\tilde{\counit}(A\Gamma)}{\longmapsto}&
\left\{
\begin{array}{lc}
  [2] T_v \bt_{S \backslash s_2} \gamma & \text{if } s_1 \in \Gamma,\\
 \sum\limits_{\{\delta:s_1 \in L(\delta)\}} \mu(\delta, \gamma )T_v \bt_{S \backslash s_2} \delta   & \text{if } s_1 \notin \Gamma,
\end{array}
\right.

\end{array}
\ee
for $v \in W^{S \backslash s_2}$.
 Then since the elements $T_v \bt_{S \backslash s_2} \gamma$ are a standard basis for  $ \H\tsr_{S \backslash s_2}A\Gamma$, the lattice $\L = A^- \H \tsr_{J'_{n-2}} \Gamma$ is sent to $\u \Ltwo$ by $\tilde{\counit}(A\Gamma)$. Now for $w \in W^{{J'_{n-2}}}$, $s_1 \notin R(w)$ implies $w \in W^{S \backslash s_2}$.  In this case,
\be \tB_{w, \gamma} \in \T_{w, \tB_{1,\gamma}} + \ui \L \xrightarrow{\tilde{\counit}(A\Gamma)} T_w \bt_{S \backslash s_2} \gamma + \Ltwo = \Ltwo.\ee
 On the other hand if $s_1 \in R(w)$, then $w = v s_1$ for $v \in W^{S \backslash s_2}$, and in this case we are assuming $s_1 \notin L(\gamma)$. Hence
\be \label{e Cvs1 mapsto}
\tB_{v s_1, \gamma} \in \T_{v, \tB_{s_1,\gamma}} + \ui \L \xrightarrow{\tilde{\counit}(A\Gamma)} T_v \bt_{S \backslash s_2} A^- \Gamma + \Ltwo = \Ltwo. \ee

\end{proof}

For the remainder of the subsection set $\L^* = A^- \H \tsr_{S \backslash s_2} \Gamma^-_{s_1}$.

\begin{theorem}\label{t bigdiagram}
The arrows in (\ref{e 5x3 diagram}) are compatible with the $W$-graph structures in the following sense.
\begin{list}{(\roman{ctr})}{\usecounter{ctr} \setlength{\itemsep}{1pt} \setlength{\topsep}{2pt}}
\item Vertical arrows take canonical basis elements to canonical basis elements or to 0.
\item The top non-zero row, on canonical basis elements, satisfies
\be\label{e mtrx1}\xymatrix@R=.16cm@C=1cm{
\tB_{w,\C_{s_1, \gamma}} \ar@{|->}[r] & \tB_{ws_1,\gamma} \ar@{|->}[r]_<<<<<<{\tilde{\counit}} & [2]\tB_{w, \gamma}, \text{   and} & (w \in W^{S \backslash s_2}) \\
\ & \tB_{w,\gamma} \ar@{|->}[r]_<<<<<<<{\tilde{\counit}} &0\mod \L^* &\ }\ee
\item The bottom non-zero row, on canonical basis elements, satisfies
\be\label{e mtrx2}
\xymatrix@R=.16cm@C=1.4cm{
\tB_{w,\C_{s_1, \gamma}} \ar@{|->}[r] & \tB_{ws_1,\gamma} \ar@{|->}[r]_<<<<<<{\tilde{\counit}} & 0, \text{ and} & **[l] (w \in W^{S \backslash s_2}) \\
& \tB_{w,\gamma} \ar@{|->}[r]_<<<<<<<{\tilde{\counit}} & \tB_{w,\gamma} &}
\ee
\end{list}
\end{theorem}
\begin{proof}
Statement (i) follows from Proposition \ref{CBSubquotientprop} and Proposition \ref{p cbsubquotientprop tsymred}.

The horizontal arrows on the left side of (\ref{e 5x3 diagram}) are understood from Theorem \ref{t twisted S2}; each is the inclusion of a cellular submodule.

To see (ii), first observe that $\Res_{S \backslash s_2} \Gamma^-_{s_1}$ and $\Lambda^-_{s_1} \subseteq \H_{S \backslash s_2} \tsr_{J'_{n-2}} A\Gamma^-_{s_1}$ (as in Theorem \ref{t twisted S2}) are isomorphic as $W_{S \backslash s_2}$-graphs.   This is clear from the remarks preceding Theorem \ref{t twisted S2} and from (\ref{e mudef}).  An isomorphism, up to a global constant, between these two objects is given by
\be A\Lambda^-_{s_1} \xrightarrow{\counit(A\Gamma^-_{s_1})}A\Gamma^-_{s_1},\ \tB_{s_1,\gamma} \mapsto [2]\gamma.\ee Therefore, tensoring $\counit(A\Gamma^-_{s_1})$ with $\H$ and applying the construction of Theorem \ref{t HY canbas exists} yields a map taking each canonical basis element to [2] times a canonical basis element. This map is the composite of the maps in the top non-zero row of (\ref{e 5x3 diagram}).

The second line of (\ref{e mtrx1}) follows from Lemma \ref{l L to L*}.

The proof of (iii) is similar to that of (ii). The $W_{S \backslash s_2}$-graphs $\Res_{S \backslash s_2} \Gamma^+_{s_1}$ and $\Lambda^+_{s_1} \subseteq \H_{S \backslash s_2} \tsr_{J'_{n-2}} A\Gamma^+_{s_1}$ are isomorphic via
\be A\Lambda^+_{s_1} \xrightarrow{\counit(A\Gamma^+_{s_1})}A\Gamma^+_{s_1},\ \tB_{1,\gamma} \mapsto \gamma. \ee
Tensoring with $\H$ yields a map taking canonical basis elements to canonical basis elements, and this map is the bottom right horizontal map of (\ref{e 5x3 diagram}).

To see the first line of (\ref{e mtrx2}), first observe that $\tB_{s_1,\gamma} = \C_{s_1} \bt \gamma
\stackrel{\counit(A\Gamma^+_{s_1})}{\longmapsto}
\C_{s_1}\gamma = 0$, with the equality by definition of the quotient $A\Gamma^+_{s_1}$. Then use the fact that any $\tB_{w,\tB_{s_1,\gamma}}$ is in $A \{ T_x \bt \tB_{s_1,\gamma} : x \in W^{S \backslash s_2}, \gamma \in \Gamma^+_{s_1} \}$ (see Theorem \ref{t twisted S2} and the preceding discussion).
\end{proof}

\begin{theorem}
\label{t main theorem}
The map $\tilde{\counit}(A\Gamma)$ (the middle right horizontal map of (\ref{e 5x3 diagram})), on canonical basis elements, satisfies

\be
\label{e tilde counit map}
\begin{array}{rclc}
\tB_{ws_1,\gamma} & \longmapsto & [2]\tB_{w,\gamma} & \text{if } s_1 \in L(\gamma), \\
\tB_{w,\gamma} & \longmapsto & 0 \mod \L^* & \text{if } s_1 \in L(\gamma), \\
\tB_{ws_1,\gamma} & \longmapsto & 0 \mod \L^* & \text{if } s_1 \notin L(\gamma), \\
\tB_{w,\gamma} & \longmapsto & \tB_{w,\gamma} \mod \L^* & \text{if } s_1 \notin L(\gamma), \\
\end{array}
\ee
where $w$ is any element of $W^{S \backslash s_2}$ (and $\L^* = A^- \H \tsr_{S \backslash s_2} \Gamma^-_{s_1}$).
\end{theorem}
\begin{proof}
The first and second line of (\ref{e tilde counit map}) follow from Theorem \ref{t bigdiagram} (ii) and the top right square of  (\ref{e 5x3 diagram}), as each vertical map in this square is the inclusion of a cellular submodule.

For the third line, apply Theorem \ref{t bigdiagram} (iii) to show that $\tB_{ws_1,\gamma} \in \H \tsr_{J'_{n-2}} A\Gamma$, going down and then right, maps to $0 \in \H\tsr_{S\backslash s_2} A\Gamma^+_{s_1}$.  Therefore (going right) $\tilde{\counit}(A\Gamma)(\tB_{ws_1,\gamma}) \in \H\tsr_{S \backslash s_2} A\Gamma^-_{s_1} \subseteq \H\tsr_{S \backslash s_2} A\Gamma$.  Combining this with Lemma \ref{l L to L*} yields the desired result.  A similar argument proves the fourth line.
\end{proof}

In a way made precise by the corollary below, the sets
\be \tsymred \Gamma^-_{s_1} \cup (\H \tsr_{J'_{n-2}} \Gamma^+_{s_1} \backslash \tsymred \Gamma^+_{s_1})
\text{ and } \tsymred \Gamma^+_{s_1} \cup (\H \tsr_{J'_{n-2}} \Gamma^-_{s_1} \backslash \tsymred \Gamma^-_{s_1})\ee
are canonical bases for $S^2_\text{red} V \tsr A\Gamma$ and $\Lambda^2V \tsr A\Gamma$, respectively, as $\u \to 0$. We therefore call these subsets of $\H\tsr_{J'_{n-2}} \Gamma$ \emph{combinatorial reduced sym} and \emph{combinatorial wedge} respectively.

\begin{corollary} \label{c approx at u=1}
After adjoining $\frac{1}{[2]}$ to $A$, there exists a $\br{\cdot}$-invariant basis $\{ c_{x, \gamma}: x \in W^{J'_{n-2}}, \gamma \in \Gamma \}$ of $\H \tsr_{J'_{n-2}} A \Gamma$ so that the transition matrix to the basis $\H \tsr_{J'_{n-2}} \Gamma$ tends to the identity matrix as $\u \to 0$, and so that under the map $\tilde{\counit}(A\Gamma)$
\be
\begin{array}{cclc}
c_{ws_1,\gamma} & \longmapsto & [2] \tB_{w,\gamma} & \text{if } s_1 \in L(\gamma), \\
c_{w,\gamma} & \longmapsto & 0 & \text{if } s_1 \in L(\gamma), \\
c_{ws_1,\gamma} & \longmapsto & 0 & \text{if } s_1 \notin L(\gamma), \\
c_{w,\gamma} & \longmapsto & \tB_{w,\gamma} & \text{if } s_1 \notin L(\gamma), \\
\end{array}
\ee
where $w$ is any element of $W^{S \backslash s_2}$.
\end{corollary}

Theorem \ref{t main theorem} and Corollary \ref{c approx at u=1} also apply with $\pi^2 A\Gamma$ replacing $A\Gamma$. There is a potential pitfall here as $\pi^2 A\Gamma$ is not the restriction of an $\H$-module to $\H_{J'_{n-2}}$. However, it is an $\H_{S \backslash s_2}$-module, since $K_0 \cap \pi^2 K_0 \pi^{-2} = S \backslash s_2$, which is all that is needed to apply the results in this subsection. Also, by \textsection\ref{ss diagram commutes} the projection $\tilde{\counit}(\pi^2 E)$ specializes to the projection  $T^2_\text{red} V\tsr E \to S^2_\text{red} V \tsr E$ at $\u =1$. Thus we can write $\H \tsr_{J'_{n-2}} \pi^2 \Gamma$ as the disjoint union of
\be
\tsymred \pi^2 \Gamma^-_{s_{n-1}} \cup (\H \tsr_{J'_{n-2}} \pi^2 \Gamma^+_{s_{n-1}} \backslash \tsymred \pi^2 \Gamma^+_{s_{n-1}} ) \text{ and } \tsymred \pi^2 \Gamma^+_{s_{n-1}} \cup (\H \tsr_{J'_{n-2}} \pi^2 \Gamma^-_{s_{n-1}} \backslash \tsymred \pi^2 \Gamma^-_{s_{n-1}} ),
\ee
which will also be called combinatorial reduced sym and combinatorial wedge.

\begin{example}
In the $W$-graph in Figure \ref{f VVe+}, combinatorial reduced sym is the lower triangular region consisting of the first $i$ entries of row $i$ for $i = 1, 2, 3$; combinatorial wedge is the upper triangular region consisting of the last $4-i$ entries of row $i$ for $i = 1, 2, 3$. For general $\Gamma$, the picture would be similar: the $W$-graph could be drawn in $n$ by $n$ chunks and combinatorial reduced sym would consist of lower triangular regions for $\gamma \in \Gamma^-_{s_1}$ and upper triangular regions for $\gamma \in \Gamma^+_{s_1}$.

In the $W$-graph in Figure \ref{f VVe+ affine}, combinatorial reduced sym is the lower triangular region consisting of the first $i-1$ entries of row $i$ for $i = 2, 3, 4$; combinatorial wedge is the upper triangular region consisting of the last $5-i$ entries of row $i$ for $i = 2, 3, 4$.

The labels ``sym'' and ``wedge'' below the trees mark the cells in combinatorial reduced sym and combinatorial wedge.
\end{example}

\section{Combinatorial approximation of $V \tsr V \tsr E \twoheadrightarrow S^2 V \tsr E$}
\label{s combinatorial approximation S2}

For this section, let $\Gamma$ be a cell of $\Gamma_W$ labeled by a tableau $T^0$. We will describe the results of \textsection\ref{s tensor V} and \textsection\ref{s decomposing VVE} in terms of cells and their tableau labels.

\subsection{}

For a tableau $P$, let $P_{r, c}$ be the square of $P$ in the $r$-th row and $c$-th column. Suppose that $P_{r_1, c_1}, \ldots, P_{r_l,c_l}$ are squares of $P$ such that $P_{r_i, c_i}$ is an outer corner of $P^{i-1} := P \backslash \{ P_{r_1, c_1}, \ldots, P_{r_{i-1}, c_{i-1}} \}$. Then referring to the sequence of tableau $P, P^1, \ldots, P^l$, we say that $P_{r_1, c_1}, \ldots, P_{r_l, c_l}$ are \emph{removed from $P$ as a horizontal strip} (resp. \emph{removed from $P$ as a vertical strip}) if $c_1 > c_2 > \dots > c_l$ (resp. $r_1 > r_2 > \dots > r_l$). Equivalently, if $P^*$ is the skew tableau of squares $\{ P_{r_1, c_1}, \ldots, P_{r_l,c_l} \}$ with $l+1-i$ in $P_{r_i,c_i}$, then $\jdt (P^*)$ is a single row (resp. column). Similarly, referring to the sequence of tableau $P^l, \ldots, P^1, P$, we say that $P_{r_l,c_l}, \ldots, P_{r_1,c_1}$ are \emph{added to $P^l$ as a horizontal strip} (resp. \emph{added to $P^l$ as a vertical strip}) if $c_1 > c_2 > \dots > c_l$ (resp. $r_1 > r_2 > \dots > r_l$).

Recall the local rules for the RSK growth diagram (see, e.g., \cite[7.13]{St}). Letting $\lambda, \mu, \nu$ be partitions with $\mu \subseteq \lambda, \nu$, we notate these local rules by
\be
\begin{array}{ll}
\G(0; \lambda, \mu, \nu) &= \left\{ \begin{array}{ll}
  \lambda & \text{if } \lambda = \mu = \nu, \\
  (\lambda_1, \ldots, \lambda_i, \lambda_{i+1} + 1, \lambda_{i+2}, \ldots) & \text{if } \lambda = \nu = (\mu_1, \ldots, \mu_i + 1, \mu_{i+1}, \ldots), \\
  \lambda \cup \nu & \text{if }  \lambda \neq \nu,
\end{array}
\right. \\
\G(1; \lambda, \mu, \nu) &= (\lambda_1 + 1, \lambda_2, \ldots) \quad \quad\quad\quad\quad\quad\quad\quad \text{ if } \lambda = \mu = \nu.
\end{array}
\ee
Here $\lambda \cup \nu$ denotes the partition whose $i$-th part is $\max(\lambda_i,\nu_i)$.

Let $a\rightarrow P$ (resp. $a \leftarrow P$) denote the column (resp. row) insertion of $a$ into $P$. For the next theorem we will use freely the descriptions of cells given in \textsection\ref{s tableau combinatorics}. The shorthand $P_>$ will be used for $P_{>c} = \jdt (P^*)$, where $P^*$ is the skew subtableau of $P$ with entries $> c$ and $c$ is the smallest entry of $P$.  In what follows we will use the somewhat redundant local sequences for cells that come from writing $\tE^1, \tE^2, \tF^2, \tF^2$ as $\H \tsr_J \H_J \tsr_J A \Gamma$, $\H \tsr_J \H_J \tsr_J \H \tsr_J \H_J \tsr_J A \Gamma$, $\H \tsr_J \H_J \tsr_{J'_{n-2}} \H_{J'_{n-2}} \tsr_{J'_{n-2}} \H_J \tsr_J A \Gamma$, $\H \tsr_{S \backslash s_2} \H_{S \backslash s_2} \tsr_{J'_{n-2}} \H_{J'_{n-2}} \tsr_{J'_{n-2}} \H_{S \backslash s_2} \tsr_{S \backslash s_2} A \Gamma$ respectively; these last two will be referred to as $\tF^2_J$ and $\tF^2_{S \backslash s_2}$ respectively.

\begin{theorem}\
\label{t combinatorial sym}
\begin{list}{(\roman{ctr})}{\usecounter{ctr} \setlength{\itemsep}{1pt}
\setlength{\topsep}{2pt}
\setlength{\leftmargin}{0pt}
}
\item The map $\H \tsr_J \unit : \tE^1 \to \tE^2$ of (\ref{e mackey2}) is given on cells
by $(T^1, T^1_>, T^0) \mapsto (T^1, T^1_>, P, T^1_>, T^0)$, where $P = 1
\to T^1_>$. In particular, \[\sh(P) = \G(1; \sh(T^1_>), \sh(T^1_>), \sh(T^1_>)).\]
\item The inverse of the map $\tE^2 \to \tF^2_J$ of (\ref{e mackey2}) is given
on cells by $(T^2, T^2_>,T^2_{>2}, T^1, T^0) \mapsto (T^2, T^2_>, P, T^1, T^0)$, where
$$\sh(P) = \G(0; \sh(T^2_>), \sh(T^2_{>2}), \sh(T^1));$$
$P$ is determined by its shape and $P_> = T^1$.
\item The isomorphism of $W$-graphs $\tF^2_J \to \tF^2_{S \backslash s_2}$ of Proposition \ref{p nested
parabolic} is given on cells by $(T^2, T^2_>, T^2_{>2}, T^1, T^0)
\mapsto (T^2, (P^2, T^2_{>2}), T^2_{>2}, (P^1, T^2_{>2}), T^0)$, where
$P^1$ is the tableau ${\tiny\young(12)}$ (resp. ${\tiny\young(1,2)}$) if
$T^0 \backslash T^1, T^1 \backslash T^2_{>2}$ are removed from $T^0$
as a horizontal strip (resp. vertical strip), and $P^2$ is the tableau
{\tiny\young(12)} (resp. {\tiny\young(1,2)}) if $T^2_> \backslash
T^2_{>2}, T^2 \backslash T^2_>$ are added to $T^2_{>2}$ as a horizontal
strip (resp. vertical strip).
\item The cells of $\tF^2_{S \backslash s_2}$ in combinatorial reduced sym
(resp. combinatorial wedge) are those with local sequences
$(T^2, (P^2, T^2_{>2}), T^2_{>2}, (P^1, T^2_{>2}), T^0)$ such that $P^2$
and $P^1$ have the same shape (resp. different shape).
\end{list}
\end{theorem}

\begin{proof}
For (i)--(iii), we will use $J = J_{n-1}$ instead of $J = J'_{n-1}$ and the comments in \textsection\ref{ss sb} to go back and forth between these conventions.

The map $\H \tsr_J
\unit$ on canonical basis elements is given in stuffed notation by
\be
(z_1, \rj{z_1}{J}, z_0) \mapsto (z_1, \rj{z_1}{J}, \rj{z_1}{J} n, \rj{z_1}{J}, z_0).
\ee
Here the $z_i$ are thought of as words so that $\rj{z_1}{J} n$ is just the word
$\rj{z_1}{J}$ with $n$ appended at the end. The map on cells is then
$(T^1, T^1_{<n}, T^0) \mapsto (T^1, T^1_{<n}, P, T^1_{<n}, T^0)$, where $P =
T^1_{<n} \leftarrow n$. Statement (i) then follows by applying the
Sch\"utzenberger involution.

For (ii), observe that the inverse of $\H \tsr_J \tau$ of (\ref{e mackey2}) is given in stuffed notation by
\be (z_2, \rj{z_2}{J},\lj{z_0}{J_{n-2}}, \lj{z_0}{J}, z_0) \mapsto (z_2, \rj{z_2}{J}, z_1, \lj{z_0}{J}, z_0),
\ee
where $z_1 = \rj{z_2}{J}^* k$ ($\rj{z_2}{J}^*$ is obtained from $\rj{z_2}{J}$ by increasing all numbers $\geq k$ by 1)
and $k$ is such that $\lj{z_0}{J} = \lj{z_0}{J_{n-2}}^* k$ ($\lj{z_0}{J_{n-2}}^*$ is obtained from $\lj{z_0}{J_{n-2}}$ by increasing all numbers $\geq k$ by 1).  Thus if $(T^2_{<n})^* := P(\rj{z_2}{J}^*)$ and $(T^2_{<n-1})^* := P(\rj{z_0}{J_{n-2}}^*)$, then
\be \label{e main ii}
P:=P(z_1) = (T^2_{<n})^* \leftarrow k, \text{ and } T^1:=P(\lj{z_0}{J}) = (T^2_{<n-1})^* \leftarrow k.
\ee
Note that $k \neq n$ and $\lj{z_0}{J_{n-2}} = \rj{z_2}{J_{n-2}}$ imply $(T^2_{<n})^* \backslash (T^2_{<n-1})^*$  is a square containing an $n$.  The element $k$ inserts in these tableau exactly the same way, except that the final step of $(T^2_{<n})^* \leftarrow k$ may bump the $n$ down one row; this case corresponds exactly to the case $\sh((T^2_{<n})^*) = \sh(T^1)$.

Statement (iii) is really two separate statements, one for a bijection
of local sequences corresponding to $\Res_{J'_{n-2}} \Res_J A \Gamma
\cong \Res_{J'_{n-2}} \Res_{S \backslash s_2} A \Gamma$, and one for a
bijection of local sequences corresponding to $\H \tsr \H_J \tsr \ce
\cong \H \tsr \H_{S \backslash s_2} \tsr \ce \ (\ce \text{ some cell of }
\Gamma_{W_{J'_{n-2}}})$. The first bijection follows from \cite[Lemma
7.11.2]{St} (This includes the statement that if $P$ is a tableau and $j
\leq k$, then the square $(P \leftarrow j) \backslash P$ lies strictly
to the left of $((P \leftarrow j) \leftarrow k) \backslash (P \leftarrow
j)$. We also need that if $j > k$, then the square $(P \leftarrow j)
\backslash P$ lies weakly to the right of $((P \leftarrow j) \leftarrow
k) \backslash (P \leftarrow j)$, which is similar.) The second bijection
is the definition of adding as a horizontal or vertical strip in the
case that $J = J_{n-1}$.

To see (iv), observe that the local labels of the cells of $\Res_{S
\backslash s_2} \Gamma^+_{s_1}$ (resp. $\Res_{S \backslash s_2}
\Gamma^-_{s_1}$) are of the form $({\tiny\young(12)}, T^2_{>2})$ (resp.
$({\tiny\young(1,2)}, T^2_{>2})$); the local labels of the cells of
$\Lambda^+_{s_1}$ (resp. $\Lambda^-_{s_1}$) are of the form
$({\tiny\young(12)}, T^2_{>2})$ (resp. $({\tiny\young(1,2)},
T^2_{>2})$), where $\Lambda = \H_{S \backslash s_2} \tsr_{J'_{n-2}} \Gamma$.
\end{proof}

\begin{example}
Suppose $T^0 = {\tiny\young(123,46,5)}$. On the left is the local sequence of a cell of $\tE^2$ (reading from left to right, ignoring the bottom middle tableau) and the local sequence of the corresponding cell of $\tF^2_J$ (reading from left to right, ignoring the top middle tableau). The tableau are arranged this way to match an RSK growth diagram picture. Above the tableau are the coordinates of the stuffed notation for a canonical basis element in this cell.

On the right is the local sequence of the corresponding cell of $\tF^2_{S \backslash s_2}$.

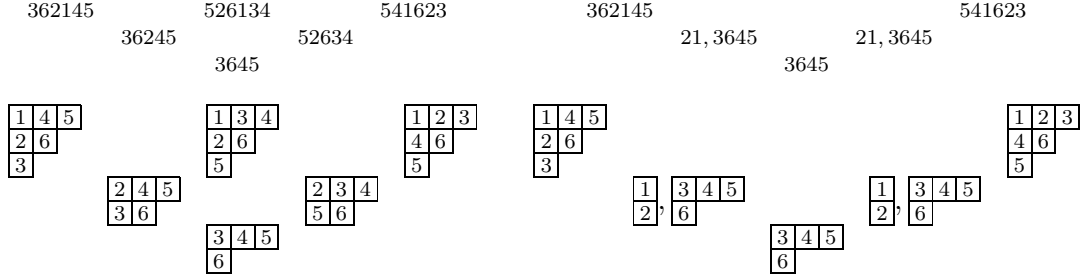
\begin{pspicture}(10,3.9){
\psset{unit=1cm}

\rput(2,3){
$\tiny\begin{array}{ccccc}
362145 && 526134 && 541623 \\
& 36245 && 52634 & \\
&& 3645 &&
\end{array}$
}
\rput(2,1){
$\begin{array}{ccccc}
\tiny\young(145,26,3) & & \tiny\young(134,26,5) & & \tiny\young(123,46,5)\\
& \tiny\young(245,36) && \tiny\young(234,56) & \\
&& \tiny\young(345,6) &&
\end{array}$}

\rput(9.5,3){
$\tiny\begin{array}{lcccr}
362145 && && 541623 \\
& 21,3645 && 21,3645 & \\
&& 3645 && \\
\end{array}$}

\rput(9.5,1){
$\begin{array}{ccccc}
\tiny\young(145,26,3) & &  & & \tiny\young(123,46,5)\\
& \tiny\young(1,2),\young(345,6) && \tiny\young(1,2),\young(345,6) & \\
&& \tiny\young(345,6) &&
\end{array}$}
}
\end{pspicture}
\end{example}

\subsection{}
For the $W$-graph version of tensoring with $V$ coming from the affine Hecke algebra, we have a similar theorem.  Let
$\G'(a;\lambda, \mu, \nu) = (\G(a;\lambda',\mu',\nu'))'$, where $'$ of a partition denotes its transpose.  Let $\hE^1, \hE^2, \hF^2_J, \hF^2_{S \backslash s_2}$ be defined analogously to $\tE^1, \tE^2, \tF^2_J, \tF^2_{S \backslash s_2}$. More precisely, $\hE^1$ and $\hE^2$ will use three- and five-term local sequences as in Examples \ref{ex affine insert} and \ref{ex hE2 cell}; $\hF^2_J$ refers to $\H \tsr_J \H_J \tsr_{J'_{n-2}} \pi^2(\H_{J_{n-2}} \tsr_{J_{n-2}} \H_{J_{n-1}} \tsr_{J_{n-1}} A \Gamma)$ with a six-term local sequence as in Example \ref{ex hE2 cell}, and $\hF^2_{S \backslash s_2}$ refers to $\H \tsr_{S \backslash s_2} \H_{S \backslash s_2} \tsr_{J'_{n-2}} \pi^2(\H_{J_{n-2}} \tsr_{J_{n-2}} \H_{S \backslash s_{n-2}} \tsr_{S \backslash s_{n-2}} A \Gamma)$ also with a six-term local sequence.

\begin{theorem}\
\label{t combinatorial sym affine}
\begin{list}{(\roman{ctr})}{\usecounter{ctr} \setlength{\itemsep}{1pt}
\setlength{\topsep}{2pt}
\setlength{\leftmargin}{0pt}
}
\item The inverse of the map $\hE^2 \to \hE^1$ of (\ref{e mackey2 affine}) is given on cells by $(T^1, T^1_>, T^0) \mapsto (P^2, P^2_>, P^1, T^1_>, T^0)$, where $P^1$ is determined by $\sh(P^1) = \G'(1; \sh(T^1_>), \sh(T^1_>), \sh(T^1_>))$ and the entries in $P^2, P^2_>$ have the same relative order as those in $T^1, T^1_>$.
\item The map $\hF^2_J \to \hE^2$ of (\ref{e mackey2 affine}) is given on cells by $(T^2, T^2_>, T^2_{>2}, \pi^{-2} T^2_{>2}, T^1, T^0) \mapsto (P^2, P^2_>, P^1, P^1_>, T^0)$, where $P^1$ is determined by $\sh(P^1) = \G'(0; \sh(T^2_>), \sh(T^2_{>2}), \sh(T^1))$ and the entries of $P^2, P^2_>, P^1_>$ have the same relative order as those in $T^2, T^2_>, T^1$.
\dikte=0.2pt
\item The isomorphism of $W$-graphs $\hF^2_J \cong \hF^2_{S \backslash s_2}$ of Proposition \ref{p nested parabolic} is given on cells by \[ (T^2, T^2_>, T^2_{>2}, \pi^{-2} T^2_{>2}, T^1, T^0) \mapsto (T^2, (P^2, T^2_{>2}), T^2_{>2}, \pi^{-2} T^2_{>2}, (\pi^{-2} T^2_{>2}, P^1), T^0), \] where $P^1$ is the tableau ${\tiny\begin{Young}{n-1} & n\cr \end{Young}}$ (resp. ${\tiny\begin{Young}{n-1} \cr n\cr \end{Young}}$)
    if $T^0 \backslash T^1, T^1 \backslash \pi^{-2} T^2_{>2}$ are removed from $T^0$ as a horizontal strip (resp. vertical strip), and $P^2$ is the tableau {\tiny\young(12)} (resp. {\tiny\young(1,2)}) if $T^2_> \backslash T^2_{>2}, T^2 \backslash T^2_>$ are added to $T^2_{>2}$ as a horizontal strip (resp. vertical strip).
\item The cells of $\hF^2_{S \backslash s_2}$ in combinatorial reduced sym (resp. combinatorial wedge) are those with local sequences \[ (T^2, (P^2, T^2_{>2}), T^2_{>2}, \pi^{-2} T^2_{>2}, (\pi^{-2} T^2_{>2}, P^1), T^0) \] such that $P^2$ and $P^1$ have the same shape (resp. different shape).
\end{list}
\end{theorem}
\begin{proof}
  Similar to that of Theorem \ref{t combinatorial sym}, the main difference being for (ii): after applying the Sch\"utzenberger involution, the analogous statement to (\ref{e main ii}) is with column insertions instead of row insertions.
\end{proof}

\begin{example}
\label{ex hE2 cell}
The local sequence for a cell of $\hE^2$ (top) and the local sequence for the corresponding cell of $\hF^2_J$ (bottom).

\begin{center}
$
\begin{array}{ccccc}
\tiny\young(\mone 12,36,4) & & \tiny\young(\mone 14,26,3) & & \tiny\young(145,26,3)\\
& \tiny\young(12,36,4) && \tiny\young(14,26,3) & \\
&& &&
\end{array}$

$\begin{array}{cccccc}
\tiny\young(123,46,5) & & & & & \tiny\young(145,26,3)\\
& \tiny\young(23,46,5) &&& \tiny\young(14,25,3) & \\
&& \tiny\young(36,4,5)& \tiny\young(14,2,3)&&
\end{array}
$
\end{center}
\end{example}

\begin{corollary}
\label{c main}
Theorem \ref{t combinatorial sym} gives a partition of the cells of $\tE^2$ into three parts: the non-reduced part corresponding to the image of (i), the inverse image of combinatorial reduced sym under (ii), and the inverse image of combinatorial wedge under (ii). Similarly, Theorem \ref{t combinatorial sym affine} gives a partition of the cells of $\hE^2$ into three parts: the non-reduced part corresponding to the inverse image of (i), the image of combinatorial reduced sym under (ii), and the image of combinatorial wedge under (ii).
\end{corollary}

\begin{remark}
There is an obvious bijection between the cells of $\tE^2$ and $\hE^2$ obtained by taking a local sequence ${\bold \ce}$ of a cell of $\tE^2$ to the cell of $\hE^2$ with the same sequence of shapes as those of ${\bold \ce}$. Under this bijection, the cells of any of the three parts of $\tE^2$ coming from Corollary \ref{c main} do not match the corresponding parts of $\hE^2$.
\end{remark}

\section{Future work}

There are some natural questions to ask about the inducing $W$-graphs construction that, as far as we know, remain unanswered.  One question is whether the edge weights $\mu$ of the $W_J$-graph $\Gamma$ being nonnegative implies the same for the coefficients $\tP_{x,\delta,w,\gamma}$ of (\ref{e mudef}) or for the structure constants $h_{x,\mathbf{y},\mathbf{z}}$, defined by $\C_{x} \tB_{\mathbf{y}} = \sum_{\mathbf{z}} h_{x,\mathbf{y},\mathbf{z}} \tB_{\mathbf{z}}$, $x \in W,\ \mathbf{y},\mathbf{z} \in W^J \times \Gamma$.  Our computations in the case $W = \S_n$ are consistent with these positivity conjectures, but we have not investigated the inducing $W$-graphs construction outside this case. Presumably these should be provable in the special case that $(W,S)$ is crystallographic, and $\Gamma$ is the iterated induction of Hecke algebras of crystallographic Coxeter groups, by the same methods used to show the non-negativity of the usual Kazhdan-Lusztig polynomials for such $W$.

Another question concerns the partial order on the cells of $\tE^{d-1} = \HJJH$.  For type $A$, we have stated Conjecture \ref{c dominance}.  For general type, we might hope to extend Lusztig's $a$-invariant
to the induced $W$-graph setting. In particular, each cell of $\tE^{d-1}$  is contained in a cellular subquotient isomorphic to    $\Gamma_{W_1}$ (Theorem \ref{t HJHcells}), so inherits  an $a$-invariant from this isomorphism; a natural question is whether $\mathbf{z} \leq_\Lambda \mathbf{z'}$ and $\mathbf{z}$, $\mathbf{z'}$ in different cells implies $a(\mathbf{z}) > a(\mathbf{z'})$, where $\Lambda$ is the $W_1$-graph structure on $\tE^{d-1}$.  In \cite{G2}, Geck shows a slightly weaker version of this statement in the case $\tE^{d-1} = \Res_{J_1} \H_2$, $d=2$ and $W_2$ crystallographic and bounded in the sense of \cite[1.1 (d)]{L}.  It seems likely that a similar proof will work for the general case, with all Coxeter groups crystallographic and bounded.

In the forthcoming paper \cite{B}, we look at the partial order on the cells of $\Res_\H \pH \tsr_{\H} e^+$.  It appears that there are other important invariants  besides the $a$-invariant and dominance order that put restrictions on this partial order.

We have put much effort into extending the results of \textsection\ref{s tensor V}-\textsection \ref{s combinatorial approximation S2} to higher symmetric powers of $V$ and have had only partial success. In a way, this is the subject of the forthcoming paper \cite{B}, however this focuses more on the extended affine Hecke algebra and less on iterated restriction and induction.

\section*{Acknowledgments}
This paper would not have been possible without the generous advice from and many detailed discussions with Mark Haiman.  I am also grateful to John Stembridge for pointing out references \cite{HY1}, \cite{HY2}, and to Michael Phillips, Ryo Masuda, and Kristofer Henriksson for  help typing and typesetting figures.

\end{document}